\documentclass[11pt]{amsart}
\usepackage[unicode]{hyperref}
\usepackage{a4,mathrsfs}

 \usepackage{charter}
\usepackage{eulervm} 

\usepackage{color}
\usepackage[color,notref,final]{showkeys}
\definecolor{refkey}{rgb}{1,0,0}
\definecolor{labelkey}{rgb}{0,0,1}  

\usepackage{amsmath,amsfonts,amssymb,amsthm,amsxtra,latexsym, amscd,amsthm}
\usepackage[all]{xy}

 \newtheorem{theorem}[subsubsection]{Theorem}

\newtheorem{lemma}[subsubsection]{Lemma}
\newtheorem{proposition}[subsubsection]{Proposition}

\theoremstyle{definition}
\newtheorem{example}[subsubsection]{Example}
\newtheorem{definition}[subsubsection]{Definition}
\newtheorem{remarks}[subsubsection]{Remarks}

\newtheorem{questions}[subsubsection]{Questions}

\def\to{\longrightarrow}
\def\Diff{\mathscr D}

\def\str{{\sf str}}

\def\Hom{{\rm Hom}}
 \def\Spf{{\sf Spf}}
 
 \newcommand{\surj}{\twoheadrightarrow}
\newcommand{\inj}{\hookrightarrow}
\def\id{{\sf id}}
 \def\im{{\sf im}}
 \def\ker{{\sf ker}}

\def\mod{{\sf Mod}}
\def\modo{{\sf Mod}^\text{o}}
\def\modf{{\sf Mod}_{\rm f}}
\def\repf{{\sf Rep}_{\rm f}}
\def\repo{{\sf Rep}^{\rm o}}
\def\comodf{{\sf Comod}_{\rm f}}
\def\comod{{\sf Comod}}
\def\Comod{\comod}
\def\Mod{\mod}
\def\comodo{{\sf Comod}^{\rm o}}
\def\coend{{\sf Coend}}

\begin{document}
\title{Tannakian duality over Dedekind rings and applications}
\author{Nguyen Dai Duong}
 
\email{nguyendaiduongqn@yahoo.com.vn}
\author{Ph\`ung H\^o Hai}
\email{phung@math.ac.vn }

\address[Nguyen Dai Duong \& Ph\`ung H\^o Hai]{Institute of  Mathematics,  Vietnam Academy of Science and Technology, Hanoi, Vietnam}

\dedicatory{Dedicated to H\'el\`ene Esnault, with admiration and affection}

\thanks{This research is funded by Vietnam National Foundation for  Science and Technology Development (NAFOSTED) under grant number 101.04-2016.19. Part of this work has been carried out when the second named author was visiting the Vietnam Institute for Advanced Study in Mathematics.}
\begin{abstract} 
We establish a duality between flat affine group schemes and rigid tensor categories equipped with a 
neutral fiber functor (called Tannakian lattice), both defined over a Dedekind ring. We use this duality and the known Tannakian duality due to Saavedra to 
study morphisms between flat affine group schemes. Next, we apply our new duality to the   
category of stratified sheaves on a smooth scheme over a Dedekind ring $R$ to define the relative differential fundamental group scheme of the given scheme and compare the fibers of this group scheme with the fundamental group scheme of the fibers. When $R$ is a complete DVR of equal 
characteristic we show that this category is Tannakian in the sense of Saavedra.
\end{abstract}

\maketitle
\parskip5pt

\section*{Introduction}

Tannakian duality for group scheme over a field was studied by Saavedra \cite{SR72}. The duality in the neutral case, as shown by Saavedra, is a dictionary between $k$-linear abelian rigid tensor categories equipped with a fiber functor to the category of $k$-vector spaces and affine group schemes over $k$. The duality consists of two parts:
\begin{itemize}
\item[-] The reconstruction theorem which recovers a group scheme from a neutral Tannakian category $(\mathcal T,\omega:\mathcal T\to {\sf Vect}(k))$, as the group of automorphisms of $\omega$ preserving the tensor product, the Tannakian group of $(\mathcal T,\omega)$.
\item[-] The presentation (or description) theorem which claims the equivalence between the original category $\mathcal T$ and the representation category of the Tannakian group of $\mathcal T$.\end{itemize}

Saavedra also extended this result to the non-neutral case - when the fiber functor goes to a more general category of coherent sheaves over a $k$-scheme. A complete proof of this theorem was given by Deligne in \cite{DP}.

An important application of Tannakian duality is to define various fundamental group schemes.
Let $X$ be a scheme over a field $k$. There are certain abelian tensor categories associated to $X$. 
For example, if $k$ is perfect, $X$ is reduced and connected, M. Nori introduced the category of essentially finite bundles; if $X$ is smooth and $k$ has characteristic zero, one has the category of flat connections on $X$; 
if $k$ has positive characteristic, one has the category of stratified bundles (i.e. $\mathcal O_X$-coherent modules equipped with the action of the sheaf $\Diff_{X/k}$ of algebras of differential operators on $X$). 
Given a $k$-rational point $x$ of $X$, the functor taking fibers at $x$ makes the above categories Tannakian category and Tannakian duality yields the corresponding affine group scheme, which is usually called the fundamental group scheme of $X$. Tannakian duality is also used as an alternative approach to the Picard-Vessiot theory of linear differential equations.

Let now $f:X\to S:={\sf Spec} R$  be a smooth morphism, where $R$ is a Dedekind ring. We are interested in the category of modules over  $\Diff(X/S)$, which are coherent as 
$\mathcal O_X$-modules, where $\Diff(X/S)$ denotes the sheaf of algebras 
of differential operators on $X/S$. Such a sheaf will be 
called stratified sheaf over $X/S$. The category ${\sf str}(X/S)$ of stratified sheaves on $X/S$ is an abelian tensor 
category, in which an object is rigid if it is locally free as an 
$\mathcal O_X$-module. Assume that $f$ admits a section $\xi:S\to X$. Then the functor  $\xi^*$ 
provides us a fiber functor for ${\sf str}(X/S)$. It is then 
natural to ask, if there exists a generalization of Tannakian duality to this case.

In fact, Tannakian duality over Dedekind rings has already been considered by Saavedra in \cite[II.2]{SR72}. Saavedra gave a condition for 
an abelian category equipped with an exact faithful  functor 
to the module category (over a given Noetherian ring) to be 
equivalent to the comodule category over the coalgebra 
reconstructed from this functor.  This duality is then
developed to a Tannakian duality for flat affine group 
schemes over a Dedekind ring.

In many examples, the involving category may not
satisfy all the properties in Saavedra's definition.  
One therefore wants to reconstruct an affine group 
scheme from a smaller part of its representation category, for instance, from rigid representations (i.e. representations in finite projective modules). 
There are at least two approaches to this problem, one by  Wedhorn \cite{We}, and the other by Brugui\'eres \cite{Br} following an idea of M. Nori, the latter one has been used by dos Santos \cite{JSa} to define the Galois differential groups of a relative stratified bundle. 
Wedhorn's approach is similar to Saavedra's approach. He introduced the notion of Tannakian lattice and reconstructed a group scheme from such a lattice. 
However Wedhorn has not 
established a full duality; he did not show that a Tannakian lattice is equivalent to the category of finite projective representations of a flat affine group scheme. 

In Section \ref{Sect1} we first recall the definition of the 
representation category of a flat affine group scheme over a Dedekind ring. 
A special property of such a category, as 
noticed by Serre \cite{Se}, is that any object can be 
represented as a quotient of a projective one (over the base ring). Saavedra uses this feature to characterize 
these representation categories (Theorem \ref{thSa}). A self-contained proof of 
this theorem will be given in the Appendix. 

In Section 
\ref{Sect2} we 
propose a new definition of neutral Tannakian lattices.
Saavedra's method is 
used to establish a full 
duality between neutral Tannakian lattices and flat affine group 
schemes over a Dedekind ring (Theorem \ref{main}). 
As a corollary we show that a Tannakian lattice is contained in a unique (up to equivalence) abelian envelope (Proposition \ref{torsor}). We also construct a torsor of isomorphisms between two fiber functors.

The Tannakian duality is used in Section \ref{Sect3} to 
study properties of homomorphisms of flat coalgebras over a Dedekind ring. 
 We first introduce the 
notions of special coalgebra homomorphisms and special subcoalgebras (Definition \ref{def_special}). For coalgebras over a field, one has the local 
finiteness property: each coalgebra is the 
union of its finite dimensional subcoalgebras. The situation 
becomes more complicated for flat coalgebras over 
a Dedekind ring. A flat coalgebra over a Dedekind ring is still 
the union of its finite subcoalgebras, but one cannot choose the subcoalgebras to be at the same time finite and saturated as submodules. We introduced the notion of special locally finite coalgebras over a Dedekind ring, which are those representable as union of finite saturated subcoalgebras. We show that the coordinate ring of a flat group scheme, the generic fiber of which is  reduced and connected,  is specially locally finite (Proposition 
\ref{smooth_lf}).  
In the second part of this section we provide conditions for a coalgebra homomorphism to be (specially) injective or surjective (Propostions \ref{hom_inj}, 
\ref{hom_sur}).

In Section \ref{Sect4} the results on homomorphisms of flat coalgebras are used to study 
homomorphisms of flat group schemes: to characterize
closed immersions and the faithful flatness.
In the last part of Section \ref{Sect4},  we give a criterion for the exactness of sequences of 
homomorphisms of flat affine group schemes over Dedekind 
rings in terms of Tannakian duality. Such a criterion for group schemes over a field has 
proved to be a useful tool, see \cite{EH06, EHS08, EH08, Hai13}. 
Some applications of the results in this section will be presented in a subsequent 
paper \cite{DHS}.

In Section \ref{Sect5} we  apply the Tannakian duality to the 
category $\str(X/R)$ of stratified sheaves over a smooth scheme $X/R$, 
where $R$ is a Dedekind ring. We show that the subcategory $\str^{\rm o}(X/R)$ of relative stratified bundles is a Tannakian lattice (assuming the existence of an $R$-point of $X$). We then compare the fibers of the relative fundamental group and the fundamental group of the fibers.
When $R$ is a complete local 
discrete valuation ring of equal characteristics, we show 
that the category of all stratified sheaves $\str(X/R)$ is a Tannakian category over $R$.  

In the Appendix, for the sake of the reader, we recall Saavedra's proof of Tannakian duality for flat affine group schemes over a Dedekind ring.

\section{Preliminaries}\label{Sect1}
 In this article $R$ will denote a {\em Dedekind ring}. The fraction field of $R$ is denoted by $K$, a residue field will be denoted by $k$. 
 The category of $R$-modules is denoted by ${\rm Mod}(R)$, 
its full subcategory of finite modules is denoted by ${\rm Mod}_\text{f}(R)$ and the full subcategory of finite projective modules is denoted by ${\rm Mod}^\text{o}(R)$. 
We shall intensively use the facts that over a Dedekind ring, torsion free modules are  flat and  finitely generated flat  modules are projective.
 The  tensor product of $R$-modules, when not explicitly indicated, is understood as the 
 tensor product over $R$.

\subsection{Flat affine group schemes} 
Let $G$ be a flat affine group scheme over $R$. The coordinate ring of $G$, $R[G]$, is usually denoted by $L$ for short. Thus $L$ is an $R$-flat (commutative) Hopf algebra.  

\subsubsection{}By definition, a $G$-{\em module} (or a $G$-{\em representation}) is the same as  a {\em right} $L$-comodule, i.e., an $R$-module $M$ equipped with a coaction of $L$:
$$\rho:M\to M\otimes_RL;\quad (\rho\otimes \id)\rho=(\id\otimes\Delta)\rho,\quad (\id\otimes\varepsilon)\rho=\id,$$
where $\Delta:L\to L\otimes L$ denotes the coproduct and 
$\varepsilon:L\to R$ denotes the counit of $L$.  
 
The flatness of $L$ implies that the category $\comod(L)$ of (right)
$R$-modules equipped with a coaction of $L$ is an $R$-
linear abelian category. We call an $L$-comodule $M$ {\em 
finite} if it is finite as an $R$-module and {\em finite 
projective} if it is also projective over $R$. The full 
subcategory of finite $L$-comodules is denoted by  
$\comod_f(L)$ and the full subcategory of finite projective 
$L$-comodules is denoted by $\comod^o(L)$. 
Each $L$-comodule is the union of its finite subcomodules (see Appendix).

\subsubsection{} $\comod(L)$ is a tensor category with respect to the tensor product over $R$. The unit object is  $R$ equipped with the trivial coaction of $L$: $R\to R\otimes L=L$, $1\mapsto 1\otimes 1$. 
The subcategory $\comod^o(L)$ is rigid, i.e., each objects possesses a dual.
\subsubsection{}
More general, a comodule $J$ is said to be trivial if the 
coaction maps any element $m$ to $m\otimes 1\in J\otimes L$. A finite trivial comodule is thus a quotient 
of the (trivial) comodule $R^n$. By taking  duality we see that a rigid trivial
comodule is a subcomodule of $R^n$. For a  
comodule $V$, the maximal trivial subcomodule $V^{\rm triv}$ of  consists of elements $v$ such that 
$\rho(v)=v\otimes 1$.

\subsubsection{}\label{sect_cf}Let $V$ be a finite projective $L$-comodule. The coaction $\rho:V\to V\otimes L$ induces a map
$${\sf Cf}: V^\vee\otimes V\to L, \quad \varphi\otimes m\mapsto \sum\varphi(m_i)m'_i, \quad \varphi\in V^\vee, m\in V, \Delta(m)=\sum_im_i\otimes m_i'.$$  
This map can be considered either as a homomorphism of $L$-comodules, 
where $L$ coacts on itself by the coproduct and coacts on $V^\vee\otimes V$ by the action on the second tensor component, or as a homomorphism of coalgebras. 
The image of this map, denoted by ${\sf Cf}(V)$, is called the coefficient space of $V$. 
Since $L$ is flat, it is a subcoalgebra of $L$, i.e., $\Delta({\sf Cf}(V))\subset {\sf Cf}(V)\otimes {\sf Cf}(V)$ (cf. Lemma \ref{lem_intersect}~(i)). 
The coaction of $L$ on $V$ factors through the coaction of ${\sf Cf}(V)$ on $V$.

We note that our requirement here for a subcoalgebra is weaker than in some other literatures, e.g. \cite{Ha}, where a subcoalgebra is a special subcoalgebra in our sense (see. \ref{def_special} for definition).

\subsubsection{}\label{s.104}
Following dos Santos, we call a subcomodule $U$ of $V$ {\em special} if $U$ is saturated in $V$, i.e., $V/U$ is $R$-flat. For instance, the coaction $V\to V\otimes L$ can be considered as a comodule map (where $L$ coacts on the target by the coaction on itself) . It splits as an $R$-module map by means of the counit $\varepsilon:L\to R$ of $L$. Hence, if $V$ is $R$-flat, $V$ is a special submodule of $V\otimes L$.

If $V$ is finite projective, then ${\sf Cf}(V)$ contains both 
${\sf Cf}(U)$ and ${\sf Cf}(V/U)$, \cite[Lem.~9]{JSa}. 
In deed, this inclusion property is a local property, so we can assume that $R$ is a DVR, then $V, U$ and $V/U$ are all free over $R$. Hence we can choose a basis $(e_1,e_2,\ldots,e_n)$ of $V$ such that the first $m$ elements form a basis of $U$ and the cosets of the other elements form a basis of $V/U$. The coaction of $L$ on $V$ with respect to this basis is given in terms of a multiplicative matrix $(a^j_i)$:
$$\rho(e_i)=\sum_{j=1}^n e_j\otimes a^j_i,\quad i=1,\ldots,n.$$
Notice that the elements $a^j_i$ are uniquely determined by the basis elemnents $e_i$'s.
Now the assumption that $U$ is a subcomodule implies that $a^j_i=0$ for $i=1,\ldots,m$ and $j=m+1,\ldots, n$. It follows that the coaction of $L$ on $V/U$ with respect to the basis $(\overline{e_{m+1}},\ldots,\overline{e_n})$ is given by
$$\rho(\overline{e_k})=\sum_{l=m+1}^n
\overline{e_l}\otimes a^l_k,\quad k=m+1,\ldots,n.$$ 

For example, for an $R$-flat comodule $V$, the maximal trivial subcomodule $V^L$ of $V$ is special. 
Indeed, if we have $0\neq v=au$, $a\in R$, $u\in V$ then from the equality $\rho(au)=au\otimes 1$ 
and the flatness of $V\otimes L$, we conclude that 
$\rho(u)=u\otimes 1$. Thus $V/V^L$ is torsion 
free, hence flat. This implies that for any two finite projective
comodules $V,W$, the inclusion
$$\Hom^L(V,W)\to \Hom_R(V,W)$$
is saturated. Indeed, we have $$\Hom^L(V,W)\cong \Hom^L(R,W\otimes V^\vee)=(W\otimes V^\vee)^L\subset
W\otimes V^\vee.$$

\subsubsection{} For a finite comodule $V$, we denote by $\langle V\rangle$ the full subcategory generated by $V$, i.e., consisting of subquotients of finite direct sums of $V$.
 For a finite projective comodule $V$, denote by $\langle V\rangle_s$ the  full subcategory specially generated by $V$, i.e. consisting of special subquotients (quotients of special subobjects or special subobjects of quotients objects) of direct sums of copies of $V$.

\begin{lemma}\label{lem.106} Let $V$ be a finite projective comodule over a flat $R$-coalgebra $L$.
Then the category $\langle V\rangle_s$ is equivalent with $\comod^o({\sf Cf}(V))$ (by means of the obvious functor which is the identity functor on the underlying $R$-modules).
\end{lemma}
\begin{proof}Consider the restriction functor $\comod^o({\sf Cf}(V))\to \comod^o(L)$.   
The condition for a map $\varphi:M\to N$ to be  an $L$-comodule map reads as follows: the map
$\rho_N \varphi- (\varphi\otimes\id)\rho_M:M\to N\otimes L$ is the zero map (i.e. the outer square in the diagram below is commutative).
$$\xymatrix{
M\ar[r]^{\rho_M\quad }\ar[d]_\varphi	&M\otimes {\sf Cf}(V)\ar[r]  \ar[d]_{\varphi\otimes\id}	& M\otimes L\ar[d]^{\varphi\otimes\id}\\
N\ar[r]_{\rho_N\quad }	&N\otimes {\sf Cf}(V)\ar[r] 		 & N\otimes L.}$$
  If $N$  is flat over $R$, the horizontal map $N\otimes {\sf Cf}(V)\to N\otimes L$ in the above diagram is injective. Hence the condition for $\varphi$ to be an $L$-comodule map is the same as the condition for $\varphi$ to be a ${\sf Cf}(V)$-comodule map (which amounts to the left square to commute). Thus the restriction functor  $\comod^o({\sf Cf}(V))\to \comod^o(L)$ is fully faithful.

On the other hand, 
according to \ref{s.104}, if $W\in \langle V\rangle_s$ then ${\sf Cf}(W)\subset{\sf Cf}(V)$, 
hence $W$ is a subcomodule of ${\sf Cf}(V)$. Thus, it remains 
to show that any finite projective comodule over ${\sf Cf}(V)$ is a special subquotient of a finite direct sum of copies of $V$.

This claim holds for ${\sf Cf}(V)$ itself, as, by definition, ${\sf Cf}(V)$ is a quotient of $V^\vee\otimes V$, where 
${\sf Cf}(V)$ coacts on $V^\vee\otimes V$ by the coaction 
on the second tensor component. Further, if $M$ is a finite projective  ${\sf Cf}(V)$-comodule then  $M$ is a special subcomodule 
of $M\otimes{\sf Cf}(V)$ (cf. \ref{s.104}),  i.e. it is a special subquotient of a direct sum of copies of $V$.
The proof is complete.\end{proof}

\subsubsection{Warning} The functor $\comod_f({\sf Cf}(V))\to \comod(L)$ is faithful and exact but generally not full, see Section \ref{Sect3} below. It is not clear to us how to specify the image of this functor.

\subsubsection{}For a comodule $M$, its ($R$-) torsion part $M^{\rm tor}$ is also a subcomodule. The quotient $M/M^{\rm tor}$ is $R$-torsion free, hence flat, hence $R$-projective if it is finite over $R$.

\subsubsection{}$L$-comodules are locally finite, i.e. they are union of their subcomodules of finite type. In fact, for each finite set $S$ of a comodule $M$, there exists a subcomodule $N$, finitely generated over $R$, which contains $S$. It follows that $L$ itself is the union of its subcoalgebras, which are finite over $R$, c.f. \cite[Cor.~2]{Se}. See also Appendix \ref{ex_comod}.

\subsubsection{}Let $K$ be the fractions field of $R$. Denote $L_K:=L\otimes_RK$. Then $L_K$ is a Hopf algebra over the field $K$. If $V$ is a comodule over $L$, then $V_K:=V\otimes_RK$ is a comodule over $L_K$. For any two finite projective comodules $V, W$, the natural map
$$\Hom^L(V,W)\otimes_RK\to \Hom^{L_K}(V_K,W_K)$$
is an isomorphism. Indeed, if $f\in \Hom^{L_K}(V_K,W_K)$, then there exists $0\neq a\in R$ such that $af:V\to W$. But then we have $f=af\otimes a^{-1}.$

Conversely, let $X$ be a finite dimensional comodule   
over $L_K$. Since $X\otimes_K L_K\cong X\otimes_RL$, $X$ is an $L$-comodule. Let $(e_1,\ldots,e_n)$ be a basis of 
$X$. Then there exists a finite $L$-comodule $V\subset X$, 
which contains $(e_i,\ldots,e_n)$. 
Now $V$ is finite hence projective 
over $R$. As $V$ contains a basis of $X$, we have $V_K\cong X$. Note that $V$ is not unique, but all such $V$ has the same rank, which is the dimension of $X$ over $K$.

\subsection{Tannakian duality for abelian tensor categories}

\begin{definition}[Subcategory of definition, {cf. \cite[II.2.2]{SR72}}]\label{dominating}
Let $\mathcal{C}$ be an $R$-linear abelian category, and 
$\omega: \mathcal{C} \longrightarrow $ $\modf(R)$  be an 
$R$-linear exact faithful functor. Suppose that 
$\mathcal{C}^{\rm o}$ is a full subcategory of $\mathcal C$ such that: 
\begin{itemize}
\item[(i)] for any object  $X \in \mathcal{C}^{o}$, $\omega(X)$ is a finitely generated  projective  $R$-module;
\item[(ii)] every object of $\mathcal{C}$ is a quotient of  an object of  $\mathcal{C}^{\rm o}.$ \end{itemize}
Then $\mathcal{C}^{\rm o}$ is  called a {\em subcategory of definition} of  $\mathcal{C}$ with respect to $\omega$.
\end{definition}

This definition is motivated by the following fact, due to Serre (see \cite[Prop.3]{Se}).
For any finite $L$-comodule $E$ there exists a short exact sequence of  $L$-comodules 
$$ 0 \longrightarrow F' \longrightarrow  F\longrightarrow E \longrightarrow 0,$$
in which $F'$ and $F$ are $R$-finite projective. Thus, the subcategory $\comodo(L)$ of $R$-finite projective $L$-comodules is a subcategory of definition in $\comodf(L)$.  
 
\begin{theorem}[cf. {\cite[Thm. II.2.3.5, II.2.6.1]{SR72}}] \label{thSa} 
Let $R$ be a Noetherian ring and let $\mathcal C$ be an $R$-linear abelian category. Assume that there exist an $R$-
linear exact faithful functor $\omega:\mathcal C\to \modf(R)$ and a subcategory of definition $\mathcal C^{\rm o}$ with respect to $\omega$. Then $\omega$ 
factors through an equivalence $\mathcal C\simeq \comodf(L)$ and the forgetful functor, for some flat 
$R$-coalgebra $L$.\end{theorem}
 Although this theorem is formulated for any Noetherian ring, We don't know any examples of comodule category satisfying the conditions of Definition \ref{dominating} when $R$ is a ring of dimension larger than 1. 
 A self-contained proof of this theorem will be given in the Appendix.
\subsubsection{}\label{sect.123}
The coalgebra $L$ in the theorem above can be determined from the fiber functor $\omega$ as follows.  We claim that there is a natural isomorphism
\begin{equation}\label{eq_nat}
{\sf Nat}(\omega,\omega\otimes M)\simeq {\sf Hom}_R(L,M),\end{equation}
for any $R$-module $M$, see Appendix \ref{coend}.
 If $\mathcal C=\comodf(L)$ and $\omega$ is the forgetful functor from $\mathcal C$ to ${\sf Mod}(R)$, then the above isomorphism implies that $\coend (\omega)\simeq L$. In particular, a flat coalgebra over $R$ can be {\em reconstructed} from the category of its comodules.
 The isomorphism in \eqref{eq_nat} implies that, for any $R$-algebra $A$,
$${\sf Nat}_A(\omega\otimes A,\omega\otimes A)\simeq {\sf Hom}_R(L,A).$$
If $\mathcal C$ is a tensor category and $\omega$ is a (strict) tensor functor, then $L$ is a bialgebra and we have an isomorphism
$${\sf Nat}^\otimes_A(\omega\otimes A,\omega\otimes A)\simeq {\sf Hom}_{R-{\sf Alg}}(L,A),$$
 where ${\sf Nat}^\otimes$ denotes the set of natural transformations that are compatible with the tensor product. 
 
\subsubsection{}Assume that $\mathcal C$ is an $R$-linear 
abelian tensor category. The reader is referred to 
\cite[Sect.~1]{DM} for the notion of dual objects. An 
object is called rigid if it possesses a dual. Notice that the image of a rigid object under a tensor functor to $\modf(R)$ 
is a finite projective module. Denote by  $\mathcal C^{\rm o}$ the full subcategory of  $\mathcal C$ consisting of rigid 
objects. We say that $\mathcal C$ is {\em dominated} by $\mathcal C^\text{o}$ if each object of $\mathcal C$ is a quotient of a rigid object.
 
 \begin{definition}   A (neutral) Tannakian category  over a Dedekind ring $R$ is an $R$-linear abelian tensor category $\mathcal C$, dominated by $\mathcal C^\text{o}$, together with an exact faithful tensor functor $\omega:\mathcal C\to {\sf Mod}(R)$. In this case, $\mathcal C^\text{o}$ is a subcategory of definition in $\mathcal C$.
\end{definition}

Let $\omega^\text{o}$ denote the restriction of $\omega$ to 
$\mathcal C^\text{o}$. Then we have, for any $R$-algebra $A$,
$${\sf Nat}^\otimes_A(\omega\otimes A,\omega\otimes A)\simeq 
{\sf Nat}^\otimes_A(\omega^\text{o}\otimes A,\omega^\text{o}\otimes A)\simeq 
{\sf Aut}^\otimes_A(\omega^\text{o}\otimes A,\omega^\text{o}\otimes A).$$
For the first isomorphism see the proof of Lemma \ref{lem_restriction}, for the second isomorphism see \cite[Prop.~1.13]{DM}.

\begin{theorem}[{\cite[Thm.~II.4.1.1]{SR72}}] \label{th_duality}
Let $(\mathcal C,\omega)$ be a neutral Tannakian category over a Dedekind ring $R$.
Then the group functor $A\mapsto {\sf Aut}_A^\otimes(\omega\otimes A)$ is representable by a flat group scheme $G$ and $\omega$ factors through an equivalence between $\mathcal C$ and $\repf(G)$.
\end{theorem}

\subsection{Scalar extension}
In \cite{We}, Wedhorn proposes to establish a duality for rigid tensor 
category over a Dedekind ring. In fact, in many examples, it is easy to specify a rigid 
tensor category, but it is very difficult to determine a Tannakian category
containing the given monoidal category as a subcategory of definition. The natural
problem is to characterize intrinsically the subcategory of rigid objects in a Tannakian 
category (over a Dedekind ring). For this, Wedhorn introduces the notion of scalar extension of category to define his Tannakian lattice. Wedhorn's Tannakian lattice is not necessarily 
equivalent to the full subcategory of rigid objects in a Tannakian category. In the next section we shall provide a characterization
of such categories. In this subsection we will recall the notion of scalar extension of categories and some basis properties.

\subsubsection{} \label{s.111} 
Let $\varphi: R\to S$ be a ring homomorphism. Let $\mathcal C$ be an $R$-linear category. 
The category $\mathcal{C}_{S}$ obtained from $\mathcal{C}$ by scalar extension $\varphi$ is defined as follows.
The objects of $\mathcal{C}_{S}$ are the same as those of $\mathcal{C}$ and for two objects $X$ and $Y$ in $\mathcal{C}_{S}$ their hom-set is 
$${\rm Hom}_{\mathcal{C}_S}(X, Y) := {\rm Hom}_{\mathcal{C}}(X, Y ) \otimes_{R} S.$$
We have an $S$-linear category together with an $R$-linear functor $\varphi_*:\mathcal C\to \mathcal C_S$.

If the map $\varphi:R\to S$ is flat, the functor $\varphi_*$
preserves  monomorphisms and epimorphisms \cite[3.6]{We}.
Let $\mathcal{D}$ be another $R$-linear  category and   
$\omega: \mathcal{C} \longrightarrow \mathcal{D}$ be an $R$-linear functor. Then we have a functor 
$$\omega_{S} : \mathcal{C}_{S} \longrightarrow \mathcal{D}_{S}.$$
If $\varphi$ is flat and $\omega$ is faithful then so is $\omega_{S}$ \cite[3.7]{We}.

Assume  that $\mathcal{M}$ is an $R$-linear tensor category. 
The $R$-bilinear functor $\otimes: \mathcal{M} \times \mathcal{M} \longrightarrow \mathcal{M}$ extends to an 
$S$-bilinear functor $\otimes: \mathcal{M}_{S}  \times \mathcal{M}_{S} \longrightarrow \mathcal{M}_{S}.$
In this way $\mathcal{M}_{S}$ is an  $S$-linear tensor category. It is rigid if $\mathcal{M}$ is rigid.

\subsubsection{}\label{s.216}
Let  $L$ be a flat $R$-coalgebra. 
 Then $L_K:= L{\otimes}_R K$ is a coalgebra over $K.$ There is a natural 
functor  from   ${\rm Comod}^o(L)$ to ${\rm Comod}_f(L_K)$,  
$M\mapsto M\otimes_RK$. This induces a  functor
$$ \phi: {\rm Comod}^o(L)_K  \longrightarrow  {\rm Comod}^o(L_K).$$
This functor $\phi$ is an equivalence of abelian categories \cite[Subsection 6.4]{We}.
 Consequently, if $G$ is a flat affine group scheme over $R,$ we have an equivalence of abelian tensor categories
$${\rm Rep}^o(G)_K \equiv{\rm Rep}_f(G_K).$$

\subsubsection{Warning} Monomorphisms in 
$\Comod^o(L)$ are injective comodules maps, as the kernel 
of a map in $\Comod^o(L)$ is again in $\Comod^o(L)$ 
($L$ being flat). However an epimorphism in this category is 
not necessarily a surjective map, for instance a map 
$[a]:V\to V$ given by multiplying with a non-unit $a$ in $R$ ($a\neq 0$) is 
an epimorphism in $\Comod^o(L)$ but is not surjective. Put in another way, the 
forgetful functor from $\Comod^o(L)$ to $\Mod(R)$ does 
not preserve epimorphisms. Later we shall impose the 
condition that our fiber functor preserves images (of 
morphisms) as a replacement for the exactness.

\subsubsection{}\label{rmk.fiber}
For an $R$-linear abelian tensor category $\mathcal C$, we define the special fiber $\mathcal C_s$ at a closed point $s$ of $S:={\sf Spec}(R)$ to be the full subcategory of objects which satisfy $m_sX=0$, where $m_s$ is the maximum ideal of $R$, which determines $s$. 

For a flat affine group scheme $G$, we can identify $\repf(G_s)$ with $\repf(G)_s$,
cf. \cite[Chapt.~10]{JC}.
Indeed, let $L=R[G]$ and $L_s:=L\otimes_RR/m_s$. Then any $L_s$ comodule is an $L$-comodule in a natural ways, as we have $V\otimes_RL\cong V\otimes_{k_s}L_s$ for any $k_s$-vector space $V$ (on which $R$ acts through the map $R\to k_s$. 

If $\mathcal C$ is Tannakian then the fiber functor $\omega$ yields an equivalence between $\mathcal C_s$ and $\repf(G_s)$ where $G$ is the Tannakian group of $\mathcal C$.

One can show that $\mathcal C_s$ is equivalent to the scalar extension $\mathcal C_{k_s}$ where $k_s=R/m_s$ is the residue field of $R$ at $s$. In fact, the $R$-linearity yields the functor $X\mapsto X\otimes_Rk_s$ from $\mathcal C$ to $\mathcal C_s$ and hence a functor from $\mathcal C_{k_s}\to \mathcal C_s$ which is  an equivalence, as the base change $R\to k_s$ does not affect $\mathcal C_s$.

\section{Duality for Tannakian lattices}\label{Sect2} 
\subsection{The kernel and image of a morphism}
For the definition of a neutral Tannakian lattice, we shall need the notion of the kernel and the image of a morphism. The notion of kernels is standard in the category theory, we recall it here for the sake of the reader.
\subsubsection{Kernel} Let $\mathcal C$ be an additive category, i.e. the hom-sets are equipped with abelian group structures and the composition of morphisms is bi-additive. For a morphism $f:X\to Y$, the kernel of $f$ is the equalizer of $f$ and the zero map, i.e. the final object in the category of morphisms $h:Z\to X$ satisfying $f\circ h=0$:
$$\xymatrix{Z\ar[rd]^{h}\ar[d]_{\exists!\varphi}\\
\ker f\ar[r]_i&X\ar[r]_f&Y.}$$
The morphism $i:\ker f\to X$ is then a monomorphism.
\subsubsection{Image-factorization}\label{dn.121} The image-factorization of a morphism $f:X\to Y$ is the initial object in the category of factorizations of the form $f=g\circ h$ with $g$ being a monomorphism. 
That is,  if $f= g'h'$ is another factorization with $g'$ being a monomorphis
\begin{displaymath}
\xymatrix{ X\ar[dr]_{h'}\ar[r]^{h}\ar@/^20pt/[rr]^f &\im(f)\ar[r]^{g} \ar@{-->}[d]_{m}&Y \\
 &Z\ar[ru]_{g'}}\end{displaymath}
 then  there exists a unique morphism $m$ such that $g = g'm.$ Since $g'$ is a monomorphism, we will also have $m\circ h=h'$. If the image-factorization exists for any morphism we say that  $\mathcal C$ is a {\em category with images}.

Let $\mathcal{C,D}$ be categories with images.
  A functor $\omega: \mathcal{C} \longrightarrow \mathcal{D}$ is said to  preserve images if $\omega$ preserves the image-factorization for any morphisms in $\mathcal C$.

\begin{example} (i) Any abelian category obviously has kernels and images. An exact functor between abelian categories preserves kernels and images.

(ii) The category ${\rm Mod}^o(R)$ ($R$ being a Dedekind ring) has kernels and images, which are determined set theoretically; the forgetful functor to ${\rm Mod}_f(R)$ preserves kernels and images.

(iii) For a flat affine group scheme $G$ over a Dedekind ring $R$, the category ${\rm Rep}^o(G)$ has kernels and images, and the forgetful functor from ${\rm Rep}^o(G)$ to ${\rm Mod}_f(R)$  preserves kernels and images. 
\end{example}

\subsection{Tannakian  lattice} 
\subsubsection{}
In an additive tensor category the endomorphism ring of the unit object $I$ is a 
commutative ring. Given a commutative ring $R$, an {\em 
additive rigid tensor  category $\mathcal T$ over $R,$} is an $R$-linear
additive rigid tensor category in which the natural map  $R\to {\rm End}
(I)$ in an isomorphism, where  $I$ denotes the unit object. 

 An object $J$ in $\mathcal T$ is called trivial if there is a monomorphism $J\to I^n$. The full subcategory of trivial objects of $\mathcal T$ is denoted by $\mathcal T^\text{triv}$.

\begin{definition}\label{dn.232}
Let $R$ be a Dedekind ring. {\em A neutral 
Tannakian lattice over $R$} is an additive rigid  tensor 
category $\mathcal T$ over $R$, in which kernels and images exist, the category $\mathcal T_K$ given by scalar extension is abelian, and is equipped with an $R$-linear additive tensor functor $\omega:\mathcal T\to \Mod(R)$, satisfying the following conditions:
\begin{itemize}\item[T1)]  $\omega$ is faithful and preserves kernels and images.
\item[T2)]  $\omega$ restricted to $\mathcal T^\text{triv}$ is fully faithful. \end{itemize}
\end{definition}

\subsubsection{}\label{s.213}Since $\mathcal T$ is rigid, the image of $\omega$ is in $\Mod^o(R)$.
Since $\omega$ is faithful, the hom-sets in $\mathcal T$ are  finite flat modules over $R$, consequenlty the functor $\iota_*:\mathcal T\to\mathcal T_K$ is faithful. Therefore, we shall identify $\mathcal T$ with its 
image in $\mathcal T_K$.  In particular, morphisms in 
$\mathcal T$ are considered as morphisms in $\mathcal 
T_K$. An object $X$ of $\mathcal T$, when considered as 
object of $\mathcal T_K$, will be denoted by $X_K$, and the 
image of a morphism $f$ under $\iota_*$ will be denoted 
by the same symbol $f$.

For an element $a\in R$ and objects $X,Y$ in $\mathcal T$, 
we shall use the  symbol $[a]$ to denote the morphism 
$X\to Y$ induced by $a$. Further we shall write $[a]\circ 
f=f\circ [a]$ simply as $af$. If $f:X_K\to Y_K$ is a 
morphism in $\mathcal T_K$, then there exists $a\in R$ 
(not uniquely determined) such that $af$ is a morphism in 
$\mathcal T$. The proof of the following lemma is obvious.
\begin{lemma}\label{lem.113}
Let $g:X\to Y$ be such that $g:X_K\to Y_K$ is an isomorphism. Then there exist $a\in R$ and $h:Y\to X$, such that $h\circ g=g\circ h=[a]$. More general, given objects $X,Y,Z$ in $\mathcal T$ and morphisms $f:X\to Z$, $g:Y\to Z$. Assume that there exists $h:X_K\to Y_K$ such that $g\circ h=f$. Then there exists $a\in R$ and $h':X\to Y$, such that  $ah=h'$ and $g\circ h'= af$.
$$\xymatrix{X\ar[r]^a\ar[d]_{h'}& X\ar[d]^f\ar[dl]_h\\ 
				Y\ar[r]_g& Z}$$
\end{lemma}

\begin{lemma}\label{lem.136} Let $f:X\to Y$ be a 
morphism in $\mathcal T$. If there exists a map $\varphi:
\omega(X)\to\omega(Y)$ such that $a\varphi=\omega(f)$ 
for some $a\in R$, then there exists $g:X\to Y$ with $ag=f$ 
and $\omega(g)=\varphi$. In other words, the submodule 
$\omega(\Hom_{\mathcal T}(X,Y))$ is saturated in 
$\Hom_R(\omega(X),\omega(Y))$.\end{lemma}
\begin{proof}
First, assume that $Y=I$ -- the unit object. Consider the 
image-factorization of $f=hg:X\to J\to I$. Applying 
$\omega$ we have  the following commutative diagram
$$\xymatrix{ \omega(X)\ar[d]_\varphi\ar@{->>}[r]^{\omega(g)}\ar[rd]|{\omega(f)}& \omega(J)\ar[d]^{\omega(h)}\ar@{.>}[dl]\\
R\ar[r]_{[a]}&R}$$
Thus $\omega(h)$ is injective and its image lies in $aR$, hence it factors as $\omega(h):\omega(J)\xrightarrow{\ \psi\ } R\xrightarrow{\ [a]\ }R$. Since $\omega$ is fully faithful on trivial objects, we have $\psi=\omega(h')$, that is, $\varphi=\omega(h'\circ g)$. Hence $f=a(h'\circ g)$.

The general case follows from this case by means of the isomorphism
$${\Hom}_{\mathcal T}(X,Y)\cong {\Hom}_{\mathcal T}(X\otimes Y^\vee,I),$$
which is compatible with the fiber functor $\omega$.
\end{proof}

\begin{lemma}\label{T_K}
The functor $\omega_K:\mathcal T_K\to {\sf Vect}(K)$ is exact. Thus $(\mathcal T_K,\omega_K)$ is a neutral Tannakian category over $K$.
\end{lemma}
\begin{proof}Recall that the functor $i_*:\mathcal T\to\mathcal T_K$ preserves mono- and epimorphisms.

Let $Y_K\xrightarrow{\ f\ }Z_K$ be an epimorphism in 
$\mathcal T_K$. Multiplying $f$ with an element from $R$ if 
needed, we can assume that $f$ is in $\mathcal T$. 
Consider the image-factorization of $f$ in $\mathcal T$: 
$f=Y\xrightarrow{\ h\ }\im(f)\xrightarrow{\ g\ }Z$. Thus, in 
$\mathcal T_K$, $g$ is both epi- and monomorphism, 
hence an isomorphism. Now $\omega (h)$ is surjective 
hence so is $\omega(h)_K$, thus $\omega(f)_K$ is also 
surjective.

Let $g:X_K\to Y_K$ be the kernel of $f$ in $\mathcal T_K$. 
We can similarly assume that $g$ is in $\mathcal T$. 
Consider the image-factorization of the map $X\to \ker(f)$ induced by $g$ in $\mathcal T$:
$X\xrightarrow{\ p\ } \im(g)\xrightarrow{\ q\ } \ker(f).$ We 
have, as above, $q$ is epi- and monomorphism in $\mathcal T_K$.
Applying $\omega$ to the factorization, we obtain
$$\omega(X)\xrightarrow{\omega(p) } \omega(\im(g))\xrightarrow{\omega(q)} \omega(\ker(f)),$$
with $\omega(q)$ being invertible in ${\sf Vect}(K)$ and $\omega(p)$ being surjective. Consequently, $\omega(X)$ is the kernel of $\omega(f)$.
\end{proof}

\subsubsection{}
Let be $G$ be a flat affine group scheme over $R,$ then the category
 ${\rm Rep}^{o}(G)$ of representations of $G$ in finite projective  $R$-modules equipped with the forgetful functor   to $\Mod(R)$, is  a Tannakian lattice over $R.$

\subsection{Tannakian duality} \label{s.Tan-duality}
\subsubsection{}\label{s.311}
Let ($\mathcal{T},\omega$) be a Tannakian lattice over $R$. The discussion in \ref{sect.123} yields a coalgebra $L$ and a factorization  
 $$\xymatrix{\mathcal T\ar[rr]^\omega\ar[rd]_{\omega^L}&& \Mod(R)\\
& \comodo(L)\ar[ru]_\nu}$$
where $\nu$ is the forgetful functor. 
  Recall that $L$ is a bialgebra, $\omega^L$ is a tensor functor and we have an isomorphism
$${\sf End}^\otimes_S( \omega \otimes S)\simeq {\sf Hom}_{R-{\sf Alg}}(L,S),$$
 for any $R$-algebra $S$. 
According to \cite[Prop.~1.13]{DM}, we have
$${\rm Aut}^\otimes_S(\omega \otimes_R S)\cong{\rm End}^\otimes_S( \omega \otimes _RS)
.$$
Hence the functor ${\bf Aut}^\otimes_R(\omega):S\mapsto {\rm Aut}^\otimes_S(\omega\otimes_R S)$ is representable by $L$. That is, $L$ is a (commutative) Hopf algebra and 
$${\bf Aut}^\otimes_R(\omega)(S)\cong\Hom_\text{alg}(L,S).$$
 Thus ${\bf Aut}^\otimes_R(\omega)$ is an affine group scheme over $R$.
 
 \begin{theorem} \label{main} Let $(\mathcal T,\omega)$ be a Tannakian lattice over a Dedekind ring $R$. Then the group scheme $G={\bf Aut}^\otimes_R(\omega)$ is faithfully flat over $R$ and $\omega$  induces an equivalence 
between $\mathcal{T}$ and ${\rm Rep}^o(G)$.
\end{theorem}
 
\subsubsection{} The difficulty lies in showing that $L$ is flat. We shall use an indirect construction to prove that the Hopf algebra $L$ satisfies the claims of Theorem \ref{main}. 
Lemma \ref{T_K} shows that $\omega_K:\mathcal T_K\to {\rm Vect}(K)$ is an exact, faithful functor, i.e. a fiber functor for the abelian rigid tensor category $\mathcal T_K$. 
Thus $(\mathcal T_K,\omega_K)$ is a Tannakian category. 
The classical Tannakian duality yields a Hopf algebra $\mathcal L$ over $K$ and an equivalence $$\omega_K:\mathcal T_K\cong\Comod_f(\mathcal L).$$
For each $X\in\mathcal T$, we have
$$\omega(X)\otimes_RK=\omega_K(X_K).$$
Thus $$\omega(X)\otimes_R\mathcal L=\omega_K(X_K)\otimes_K\mathcal L.$$
Therefore we can consider $\omega(X)$ as a comodule over the $R$-coalgebra $\mathcal L$. Denote by $\mathcal L_R$ the union in $\mathcal L$ of all the coefficient spaces ${\sf Cf}(\omega(X))$ as $X$ runs in the objects of $\mathcal T$. Then $\mathcal L_R$ is a Hopf $R$-subalgebra of $\mathcal L$, which coacts on all $\omega(X)$, (moreover we have $\mathcal L_R\otimes K=\mathcal L$). Thus the functor $\omega$ factors through the functor (followed by the forgetful functor):
\begin{equation}\label{eq_LR}\omega:\mathcal T\to \comodo (\mathcal L_R).\end{equation}
 
\subsubsection{Proof of Theorem \ref{main}}
We proceed by showing that the functor $\omega$ in \eqref{eq_LR} is an equivalence of categories.

 {\em $\omega$ is fully faithful}:
By assumption, $\omega$ is faithful. 
On the other hand,  the Tannakian duality for $\mathcal T_K$ says that
$$\omega(\Hom_{\mathcal T}(X,Y))\otimes_RK=\Hom^{\mathcal L}(\omega_K(X_K),\omega_K(Y_K)).$$
Let $\varphi\in \Hom^{\mathcal L_R}(\omega(X),\omega(Y)).$ Then the above equality ensures the existence of a morphism $f\in 
\Hom_{\mathcal T}(X,Y)$ such that $\omega(f)=a\varphi$ for some $a\in R$. Now
Lemma \ref{lem.136} implies that there exists $g\in \Hom_{\mathcal T}(X,Y))$ such that $\omega(g)=\varphi$. 
That is, $\omega:\mathcal T\to\Comod^o(\mathcal L_R)$ is full.

{\em $\omega$ is essentially surjective}: Each finite projective $\mathcal L_R$-comodule $M$ is a special subquotient of $\mathcal L_R{}^r$ (direct sum of $r$ copies of $\mathcal L_R$) by means of the coaction: 
$\delta:M\to M\otimes_R\mathcal L_R,$
see  \ref{s.104}. 
 Consequently for any subcomodule $N$ in $\mathcal L_R{}^r$  containing $M$, the quotient $N/M$ is also flat. In particular, we can choose $N$ to be the subcomodule ${\sf Cf}(\omega(X))^r$ for some $X$ in $\mathcal T$. Thus we conclude that $M$ is a special subquotient of an object in the image of $\omega$.

Next, we show that each $N$ in $\Comod^o(\mathcal L_R)$ is isomorphic in this category to some $\omega(X)$. First assume that $N$ is a quotient of some $\omega(Y)$:
$$0\to M\to \omega (Y)\to N\to 0.$$
According to \ref{s.216}, 
there exists some $Z$ in $\mathcal T$ such that $N\otimes 
K\cong \omega_K(Z_K)$ in $\Comod(\mathcal L_K)$. Such 
an isomorphism yields an injective map
$N\to \omega(Z)$
in $\Comod(L)$. Consider the composition $\omega(Y)\to 
N\to \omega(Z)$. Since $\omega$ is fully faithful, this map is the 
image of a morphism $f:Y\to Z$. By assumption, $f$ has 
image in $\mathcal T$ and $\omega$ preserves it, hence we 
claim
$\omega(\im (f))=N.$
Dualizing the above sequence we have an exact sequence
$$0\to N^\vee\to \omega(Y^\vee)\to M^\vee\to 0.$$
The same discussion shows that $M^\vee\cong\omega(\im g)$ for some morphism $g$ in $\mathcal T$, thus $M\cong\omega (T)$ with $T:=(\im f)^\vee$.
Thus, all special subobjects of $\omega(X)$ are also in $\im(\omega)$. Consequently all special subquotients of $\omega$ are in $\im(\omega)$. This completes the proof of the fact that $\omega$ is an equivalence of categories between $\mathcal T$ and $\Comod^o(\mathcal L_R)$.

 {\em $\mathcal L_R$  is the coend of $\omega$}:  this follows from the equivalence just established. In deed, by the equivalence above, $\omega:\mathcal T\to \Mod(R)$ can be identified with the forgetful functor $\Comod^o(\mathcal L_R)\to \Mod(R)$. But the coend of this last functor is just $\mathcal L_R$, cf. \ref{s.311}.
  This finishes the proof of Theorem \ref{main}.\hfill $\Box$

\subsection{Torsors}

Let $(\mathcal T,\omega)$ be a Tannakian lattice. Let $S$ be 
a faithfully flat $R$-algebra. A fiber functor $\eta:\mathcal T\to \Mod^o(S)$ is defined to be an $R$-linear tensor functor 
satisfying the following conditions:
\begin{itemize}
\item[(T1)]  $\eta$ is faithful and preserves kernels and images.
\end{itemize}
In this section we construct out of this data a torsor 
${\bf Iso}^\otimes_R(\omega,\eta)$ over ${\bf Aut}^\otimes_R(\omega)$. In particular we will show that for 
different neutral fiber functors $\omega$ and $\omega'$, the 
representation categories of ${\bf Aut}^\otimes_R(\omega)$ 
and ${\bf Aut}^\otimes_R(\omega')$ are equivalent, thus 
determining a unique {\em abelian envelope} of $\mathcal T$.

\subsubsection{The abelian envelope of $\mathcal T$.} \label{s.41}
Let $(\mathcal T,\omega)$ be a neutral Tannakian lattice and let $\eta:\mathcal T\to \Mod(R)$ be another fiber functor. Let $L(\omega)$ be the coend of $\omega$ and $L(\eta)$ be the coend of $\eta$. 
Then Tannakian duality for $(\mathcal T,\omega)$ and $(\mathcal T,\eta)$ yields a functor $\varphi:\Comod^o(L(\omega))\to\Comod^o(L(\eta))$:
$$\xymatrix{& \mathcal T\ar[ld]_\omega\ar[rd]^\eta\\
\Comod^o(L(\omega))\ar[rr]_\varphi&&
\Comod^o(L(\eta)).}$$
The subcategory of trivial objects in $\comodo(L(\omega))$ is exactly the category $\modo(R)$ of finite projective $R$-modules.
The restriction of $\varphi$ to $\modo(R)$ is thus a faithful functor, preserving kernels and images and sending $R$ to itself. It is therefore a 
fully faithful functor. Thus, applying Theorem \ref{main} to $\varphi$ we 
conclude that this restriction of $\varphi$ is an equivalence of categories. We will show that $\varphi$ 
extends to an equivalence between $\Comod(L(\omega))$ and 
$\Comod(L(\eta))$.

\subsubsection{} Let
$V$ be in $\Comod^o(L(\omega))$. Denote 
$C:={\sf Cf}(V)\subset L(\omega)$. Then $\Comod^o(C)$ is 
equivalent to $\langle V\rangle_s$, cf. \ref{lem.106}. Let 
$W:=\eta(V)\in\Comod^o(L(\eta))$ and $D:={\sf Cf}(W)\subset L(\eta)$. Then $\varphi$ induces an equivalence between $\Comod^o(C)$ 
and $\Comod^o(D)$.

Let $A:=\Hom(C,R)$, then $A$ is a (non-commutative) algebra, finite projective as an $R$-module and we have an equivalence
$\Comod(C)\cong \Mod(A)$
inducing and equivalence
$$\Comod^o(C)\cong\Mod^o(A).$$
Here,  we denote by $\Mod(A)$ the category of finite left $A$-modules and by  $\modo(A)$ full subcategory   modules   which are finite projective over $R$.
Thus, denoting $B:=\Hom(D,R)$, we have and equivalence, also denoted by $\varphi$:
$$\Mod^o(A)\xrightarrow{\varphi}\Mod^o(B).$$

\begin{lemma}\label{lem.243} The functor $\varphi$ extends to an equivalence between $\Mod_f(A)$ and $\Mod_f(B)$.\end{lemma}
\begin{proof}Since $\varphi$ preserves kernels and images, it preserves exact sequences in 
$\Mod^o(A)$ (i.e sequences in $\Mod^o(A)$, exact in $\Mod(A)$.

Denote $P:=\varphi(A)$. Then $P$ is a $B-A$-bimodule (the 
action of $A$ on $P$ is induced from the right action of  $A$ 
on itself, which commutes with morphisms in $\Mod(A)$). 
 For an object $M\in\Mod^o(A)$, consider a finite resolution 
 $A^m\to A^n\to M\to 0$.  
We have an exact sequence
 $$\varphi(A^m)\to\varphi(A^n)\to\varphi(M)\to 0$$
 in $\Mod(B)$. Since $\varphi$ is additive, we have $\varphi(A^n)=P^n$, hence we have canonical isomorphism 
 $$\varphi(M)\cong P\otimes_AM.$$
Thus $\varphi$ coincides with $P\otimes_A-$ on  $\modo(A)$. 

We show that $P$ is flat over $A$.  
For a finite $A$-module $M$, consider a resolution
$$0\to N\to A^m\to M\to 0.$$
Since $N$ is $R$-projective being a submodule of $A^m$, and since $\varphi(-)=P\otimes -:\Mod^o(A)\to\Mod(B)$ preserves monomorphisms, the long exact sequence 
$$0\to {\rm Tor}^A_1(P,M)\to P\otimes N\to P^m\to P\otimes_AM\to 0$$
shows that ${\rm Tor}^A_1(P,M)=0$. Thus we have an exact functor $P\otimes-:\mod_f(A)\to\mod_f(B)$, which reduces to an equivalence $\modo(A)\to\modo(B)$. Hence it is itself an equivalence from $\mod_f(A)$ to $\mod_f(B)$.
\end{proof}

\begin{proposition}\label{torsor} The functor $\varphi:\Comod^o(L(\omega))\to \Comod^o(L(\eta))$ extends to an equivalence between $\Comod_f(L(\omega))$ and $\Comod_f(L(\eta))$.\end{proposition}
\begin{proof}
The equivalence is obtained by extending $\langle V\rangle_s$ larger and larger. 
\end{proof}
 
 \subsubsection{The torsor  ${\bf Iso}^\otimes_R(\omega,\eta)$}
 Consider now the more general situation: $\eta$ is a fiber functor $\mathcal T\to \Mod(S)$, where $S$ is an $R$-algebra.
Recall that ${\bf Iso}^\otimes_R(\omega,\eta)$ is the functor which associates to each $S$-algebra $S'$ the set ${\rm Iso}^\otimes(\omega,\eta)(S')$ of natural isomorphisms compatible with the tensor structure between the two given functors. According to \cite[Prop.~1.13]{DM}, this set is equal to the set
$ {\rm Nat}^\otimes_{S'}(S'\otimes_R\omega, S'\otimes_S\eta).$
The algebra $L(\omega,\eta)$ represents functor ${\bf Iso}^\otimes_R(\omega,\eta)$ if it satisfies
$${\rm Nat}^\otimes_{S'}(\omega \otimes_RS',\eta \otimes_SS')\cong \Hom_{S\text{-alg}}(L(\omega,\eta),S'),$$
for any $S$-algebra $S'$.
As in \ref{sect.123}, we notice that $L(\omega,\eta)$ can be defined as the $S$-module representing the functor
$$S'\mapsto {\rm Nat}_{S'}(\omega \otimes_RS', \eta \otimes_SS')\cong
{\rm Nat}_R(\omega,\eta \otimes_SS').$$
That is, we determine $L(\omega,\eta)$ by the functorial isomorphism
\begin{equation}\label{eq_Loe}
{\rm Nat}_R(\omega,\eta \otimes_SN)\cong
{\rm Hom}_S(L(\omega,\eta),N),\end{equation}
for any $S$-module $N$.
\subsubsection{} Identifying $\mathcal T$ with $\Comod^o(L)$ by means of $\omega$, we can consider $\omega$ as the identity functor and $\eta$ as a tensor functor $\Comod^o(L)\to \Mod(S)$, where  $L:=L(\omega)$. 

Let $C\subset L$ be an $R$-finite subcoalgebra, and set $A=\Hom_R(C,R)$. The discussion in the proof of Lemma \ref{lem.243} shows that the restriction of $\eta$ to 
$\Comod^o(C)=\Mod^o(A)$ extends to an exact functor $P\otimes_A-:\mod_f(A)\to\mod_f(S)$, where $P=\eta(A)$.

\begin{lemma} Consider $\omega$ and $\eta$ as functors on the category $\Mod^o(A)$, we have natural isomorphism
$${\rm Nat}_R(\omega,\eta \otimes_SN)\cong P\otimes _SN\cong \Hom_S(\eta(C),N),$$
for any $S$-module $N$.
\end{lemma}
\begin{proof} This is standard. By the functoriality, a natural transformation $\lambda:\omega\to \eta \otimes_SN$ is uniquely determined by its value at $A$, i.e. an $R$-linear map $\lambda_A:A\to  \eta(A)\otimes_SN=P\otimes_SN$, which is $A$-linear.  Indeed, this is a consequence of the following commutative diagram, for any $f$ in $\Mod^o(A)$:
$$\xymatrix{A\ar[rr]^{\lambda_A}\ar[d]_f&& P\otimes_SN
\ar[d]^{P\otimes f\otimes N}\\
M\ar[rr]_{ \lambda_M}&&(P\otimes_AM)
\otimes_SN,}$$
forcing $\lambda_M=\lambda_A\otimes_AM$.
(The $A$-linearity follows from the choice $M=A$).
Conversely, any such $A$-linear map determines a natural equivalence.

As to the last isomorphism, notice that $C$ is $R$-finite projective, hence $A$ is the dual of $C$ in $\comodo(L)$, therefore $P=\eta(A)$ is dual to $\eta(C)$ as $S$-modules. Thus we have
$$P\otimes_SN={\rm Hom}_S(\eta(C),S)\otimes_SN\cong
{\rm Hom}_S(\eta(C),N).$$
This finishes the proof.
\end{proof}

\begin{proposition} The scheme ${\bf Iso}^\otimes_R(\omega,\eta)$ is a torsor under the group scheme ${\bf Aut}^\otimes_R (\omega)$.
\end{proposition}
\begin{proof} This is a remedy of \cite[Prop.~3.2]{DM}. According to the isomorphism in \eqref{eq_Loe} and the lemma above, the $S$-module
$$L(\omega,\eta)=\varinjlim_{C\subset L(\omega)}\eta(C)$$
represents the functor ${\bf Iso}^\otimes_R(\omega,\eta)$. Note that the transition maps in the directed system are inclusions of subcoalgebras $C\inj C'$ of $L$, which give rise to injective maps $\eta(C)\to\eta(C')$.
The torsor action is obvious, loc.cit. It  remains to check that ${\bf Iso}^\otimes_R(\omega,\eta)$ is faithfully flat over $S$, i.e. to check that $L(\omega,\eta)$ is faithfully flat over $S$. As $L(\omega,\eta)$ is the direct limit of $S$-projective modules, it is flat over $S$. We show the faithfulness.

If a finite subcoalgebra $C$ contains the unit element of $L(\omega)$, then the inclusion $R\to C$ splits in $\mod(R)$ (by means of the counit). We obtain an exact sequence in $\comodo(C)$:
$$0\to R\to C\to C/R\to 0$$
giving rise to an exact sequence in $\comodo(S)$ 
(as $\omega$ preserves kernels and images)
$$0\to S\to \eta(C)\to \eta(C/R)\to 0.$$
 Passing to the limit we obtain an inclusion $S\to L(\omega,\eta)$, which is pure, as
$$L(\omega,\eta)/R\cong \varprojlim_{R\subset C\subset L}\eta(C/R)$$
is flat.
This shows that $L(\omega,\eta)$ is faithfully flat over $S$.
\end{proof}

\section{Coalgebra homomorphisms}\label{Sect3}
\subsection{Specially locally finite coalgebras}
An important property of coalgebras over a field  is the local 
finiteness: a coalgebra is the union of its finite dimensional 
subcoalgebras. This property formally generalizes to flat 
coalgebras over a Dedekind ring. Indeed, let $M\subset L$ be  a finite subcomodule, then $M$ is  
projective over $R$ and is contained in the subcoalgebra   ${\sf Cf}(M)$ of $L$ (see proof of Lemma \ref{lem_loc.fin}).

However this is not fully reflected in the Tannakian duality.  
For a subcoalgebra $C$ of coalgebra $L$ over a field $k$, 
the category $\comodf (C)$ can be identified with a full, 
exact subcategory of $\comodf(L)$, which is closed under 
taking subobjects. This is no more the case for flat 
coalgebras over a Dedekind ring as the example \ref{ex:non-loc.fin} below 
shows. The reason is that the quotient module $L/C$ may 
be non-flat over the base ring. If we try to enlarge $C$ so 
that the quotient becomes flat, we cannot guarantee that it 
remains finite over $R$. In view of Tannakian duality, this 
reflects the fact that the full abelian subcategory generated 
by a single object in a comodule category may become quite  
large.

This phenomenon is one of the main obstructions to the 
study of flat coalgebras and flat group schemes over a 
Dedekind ring. Below we will show that the coordinate ring of a reduced, connected group scheme is specially locally finite as a coalgebra (Proposition \ref{smooth_lf}).

Recall that a subcomodule $M$ of an $L$-comodule $N$ is said to be special if $N/M$ is flat over $R$; a special subquotient $M$ of an $L$-comodule $N$ is a special submodule of a quotient of $N$, or, equivalently, a quotient of a special submodule of $N$.

\begin{definition} \label{def_special} Let $L$ be an $R$-flat coalgebra.
\begin{enumerate} 
\item A comodule $N$ is said to be {\em specially locally finite} if any finite subcomodule of $N$ is contained in an $R$-finite special subcomodule;
\item A subcoagebra $C$ of $L$ is said to be {\em special} if $L/C$ is flat. A homomorphism of flat coalgebras $f:L'\to L$ is said to be special if $f(L')$ is a special subcoalgebra of $L$;
\item $L$ is said to be {\em specially locally finite}  if for any finite subcomodle $C$, there exists a finite special subcoalgebra  containing $C$\footnote{Specially locally finite coalgebras are called IFP coalgebras in \cite[I.3.11]{Ha}.}.\end{enumerate}
\end{definition}

Let $N$ be an $L$-comodule. Then $N_{\rm tor}$, the $R$-torsion submodule of $N$ is an $L$-subcomodule. Hence for any $L$-subcomodule $M$, the preimage of $(N/M)_{\rm tor}$ in $N$, denoted $M^{\rm sat}$, is an $L$-comodule. Since $R$ is a Dedekind ring, the quotient $N/M^{\rm sat}$ is flat, being torsion free. Thus $M^{\rm sat}$ is the smallest special subcomodule of $N$, containing $M$. It is called the saturation of $M$ in $N$.

\begin{lemma}\label{lem_loc.fin}
An $R$-flat coalgebra $L$ is specially locally finite if and only if it $L$ is locally finite as a comodule on itself.
\end{lemma}
\begin{proof}
Assume that $L$ is specially locally finite as a right comodule on itself. Let $C$ be a finite subcoalgebra of $L$ and $C^{\rm sat}$ the saturation. It is to show that $C^{\rm sat}$ is a subcoalgebra.

We have the filtration $C^{\rm sat}\otimes C^{\rm sat}\subset C^{\rm sat}\otimes L\subset L\otimes L$, the successive quotients of which are flat, hence  $L\otimes L/C^{\rm sat}\otimes C^{\rm sat}$ is also flat. Thus 
$$(C\otimes C)^{\rm sat}\subset C^{\rm sat}\otimes C^{\rm sat}.$$
Hence, by the definition of $C^{\rm sat}$, we have
$$\Delta(C^{\rm sat})\subset (C\otimes C)^{\rm sat}\subset C^{\rm sat}\otimes C^{\rm sat}.$$

For the converse statement we use the well-known fact that each finite subcomodule of $L$ is contained in a finite subcoalgebra of $L$. Indeed, given $M\subset L$ finite, then it is $R$-projective and the coaction $M\to M\otimes L$ 
 induces a coalgebra map (cf. \ref{sect_cf})
$$M^\vee\otimes M\to L,\quad \varphi\otimes m\mapsto \sum\varphi(m_i)m'_i, \quad \text{where }\Delta(m)=\sum_im_i\otimes m_i'.$$ Its image is   the coefficient space ${\sf Cf}(M)$ of $M$, which contains $M$. In fact, let $\varepsilon_M$ be the restriction of $\varepsilon$ to $M$, then
$\varepsilon_M\otimes m\mapsto \sum_i\varepsilon(m_i)m'_i=m$.
\end{proof}

The next example shows that there exist coalgebras which are not specially locally finite.

\begin{example}[\cite{SGA3}, Remarque 11.10.1]\label{ex:non-loc.fin} Let $0\neq x\in R$ be a non-unit element.
Let $G$ be the affine group scheme over $R$ determined by the Hopf subalgebra of $K[{\mathbf G}_a]=K[T]$:
$$R[G]:=\left\{ P\in K[T] | P(0)\in R\right\},$$
that is, $R[G]$ consists of polynomials in $K[T]$ with the constant coefficients belonging to $R$. Let $C_i$ be the subcoalgebra spanned (over $R$) by $1$ and $x^{-i}T$. Then we have
$$C_i\subsetneq C_{i+1},\quad xC_{i+1}\subsetneq C_i.$$
Thus the saturation of $C_0$ is not finite.
\end{example}

In what follows we will need some standard facts on tensor products and flat modules, cf. \cite{Bour}.
 \begin{lemma}\label{lem_intersect} Let $A\subset B$ be flat $R$-modules and $M,M_1,M_2\subset N$ be arbitrary $R$-modules. Then
 \begin{enumerate}
 \item   $M_1\otimes A\cap M_2\otimes A=(M_1\cap M_2)\otimes A$,  $M_1\otimes A+ M_2\otimes A=(M_1+ M_2)\otimes A$, 
   as subsets of $N\otimes A$;
 \item if $B/A$ is also flat, we have
$N\otimes A\cap M\otimes B=M\otimes A$
as submodules in $N\otimes B$.
\end{enumerate}
\end{lemma}

The following result is proved in more generality in \cite[I.3.11]{Ha},  we recall it here for completeness.
The reader is referred to \cite[Tag~0599]{stacks-project} for the notion of Mittag-Leffler system. 
 \begin{proposition} Let $R$ be a Dedekind ring and $L$ be an $R$-flat coalgebra. 
\begin{enumerate}\item If $L$ is specially locally finite then $L$ is Mittag-Leffler as an $R$-module. Hence, if $L$ is moreover countably generated over $R$, it is $R$-projetive.
\item If $L$ is $R$-projective, then it is specially locally finite as an $R$-coalgebra. \end{enumerate}
\end{proposition}
\begin{proof}
 (i) Let $\{C_\alpha\}$ be the directed system of finite special subcoalgebras. Then for any finite $R$-module $N$, the system
 ${\sf Hom}_R(C_\alpha,N)$ is Mittag-Leffler. In fact, each inclusion $C_\alpha\to C_\beta$ splits, as $C_\beta/C_\alpha$ is $R$-torsion free and finite, hence projective over $R$. Consequently the map
 $${\sf Hom}_R(C_\beta,N)\to {\sf Hom}_R(C_\alpha,N)$$
 is surjective. Thus by definition, $L$ is Mittag-Leffler as an $R$-module. It is well-known that a flat, countably generated, Mittag-Leffler module is projective.

 (ii) We show that any projective comodule is specially locally finite. If $N\subset M$ is a subcomodule then $N^{\rm sat}$ is the preimage of $(M/N)_{\rm tor}$ under the quotient map $M\to M/N$. Thus we have to show that $N^{\rm sat}$ is finite provided that $M$ is projective and $N$ is finite as $R$-modules. This is a pure question of $R$-modules. Embed $M$ is a free $R$-module $F$ as a direct summand. Replacing $M$ by $F$ will only enlarge $N^{\rm sat}$, thus we can assume that $M$ is free over $R$. Then, as $N$ is finite, we can find a free direct summand of $F_0$ which contains $N$. Now $F_0=F_0^{\rm sat}$ implying $N^{\rm sat}\subset F_0$, hence it is finite.
 \end{proof}
\begin{questions}
It is not known if any specially locally finite $R$-flat coalgebra is $R$-projective. Another interesting question is: which affine flat $R$-group scheme of finite type is specially locally finite.
\end{questions}
\begin{proposition}\label{smooth_lf}
Let  $G$ be a flat group scheme of finite type over a Dedekind ring $R$. Assume that the generic fiber $G_K$ is reduced and connected. Then $R[G]$ is specially locally finite as an $R$-coalgebra.
\end{proposition}
 \begin{proof}  
Let $I$ be the augmented ideal of $R[G]$, that is $I={\sf ker}(\varepsilon)$.  Since the $R[G]$ is flat over $R$, the map $R[G]\to R[G]\otimes_RK=K[G_K]$ is injective and as $K$ is flat over $R$, the augmentation ideal of $K[G_K]$ is $I_K=I\otimes_RK$. With the assumption that $G_K$ is reduced and connected, $K[G_K]$ is an integral domain. By Krull's intersection theorem $\bigcap_m(I_K)^m=0$. 

Let $M\subset R[G]$ be a finite submodule. Then there exists $m$ such that $M\otimes K\cap (I_K)^m=0$. We have $(I_K)^m=I^m\otimes K$.  Hence  
$(M\cap I^m)\otimes K=0$ implying $M\cap I^m=0$.
It follows that $M^{\rm sat}\cap I^m=0$. Indeed, if $0\neq a\in M^{\rm sat}\cap I^m$ then there exists $0\neq r\in R$ such that $ra\in M\cap I^m$, it forces $ra=0$, consequently $a=0$ as $R[G]$ is torsion free.   
Thus, the map $M^{\rm sat}\to R[G]/I^m$ is injective. But the module $R[G]/I^m$ is finite, hence so is $M^{\rm sat}$.
\end{proof}
 
On the other hand, the situation for group schemes with finite fibers turns out to be more complicated, as in the following example, communicated to us by dos Santos.
\begin{example} Let $R$ be a DVR of equal positive characteristic $p$, with uniformizer $\pi$. Let $G$ be the group scheme determined by the Hopf algebra
$$R[G]:= R[T]/(\pi T^p-T),\quad \Delta(T)=1\otimes T+T\otimes 1.$$
Then the fibers of $G$ are \'etale group scheme but $R[G]$ is not specially locally finite. Indeed, the saturation of the finite subcomodule spanned by $1$ and $T$ contains $T^{p^k}, k\geq 1,$ and is not finite.\end{example}

\subsection{Tannakian description of homomorphisms of coalgebras}
In the second part of this paragraph we will use the Tannakian duality to characterize (special) injective and surjective homomorphisms of flat coalgebras.

Let $f:L'\to L$ be homomorphism of  flat coalgebras over a Dedekind ring $R$. We denote by $\omega_f:\comodf(L')\to \comodf(L)$ the ``restriction'' functor (which considers each $L'$-comodule as an $L$-comodule by means of $f$, thus $\omega_f$ is the identity functor on the underlying $R$-modules). Then $\omega_f{}^{\rm o}:\comodo(L')\to\comodo(L)$ will denote the restriction of $\omega_f$ to the subcategory of finite projective comodules.
\begin{proposition} \label{hom_inj}
Let $f:L'\to L$ be homomorphism of  flat coalgebras over a Dedekind ring $R$. Then
\begin{enumerate}
\item $f$ is injective if and only if the functor $\omega_f{}^{\rm o}$ is fully faithful and its image in $\comodf(L)$ is closed under taking special subobjects.
\item $f$ is injective and special   if  and only if the natural functor $\omega_f{}^{\rm o}$   is fully faithful and its image in $\comodf(L)$ is closed under taking subobjects. In this case, the functor $\omega_f$ is also fully faithful and its image is closed under taking subobjects.
 \end{enumerate}
\end{proposition}
\begin{proof} We shall show the ``only if" implication for both claims at once. First notice that the full faithfulness claim in (i) is essentially proved in Lemma \ref{lem.106}. For claim (ii) the proof is almost the same: Assume that  $f$ is injective  and special.
Then the functor $\omega_f:\comodf(L')\to \comodf(L)$ can be considered as the identity functor on the underlying module category. Hence it is obviously faithful (and so is $\omega_f{}^{\rm o}$).

As mentioned in the proof of Lemma \ref{lem.106}, the condition for a map $\varphi:M\to N$ to be $L'$-comodules reads as follows:
$\rho'_N \varphi- (\varphi\otimes\id)\rho_M':M\to N\otimes L'$ is the zero map:
$$\xymatrix{
M\ar[r]^{\rho_M'}\ar[d]_\varphi	&M\otimes L'\ar[r]^{\id\otimes f}  \ar[d]_{\varphi\otimes\id}	& M\otimes L\ar[d]^{\varphi\otimes\id}\\
N\ar[r]_{\rho_N'}	&N\otimes L'\ar[r]_{\id\otimes f}		 & N\otimes L}$$
Now,  if   $L/L'$ is flat over $R$, the horizontal map $\id\otimes f:N\otimes L'\to N\otimes L$ in the above diagram is injective. Hence the map 
$\rho'_N \varphi- (\varphi\otimes\id)\rho'_M$ is zero if and only if its composition with $\id\otimes f$, which is $\rho_N \varphi- (\varphi\otimes\id)\rho_M:M\to N\otimes L$, is zero. Thus $\omega_f$ is full.

We show the closedness under taking (special) subquotients.
For $(M, \rho') \in $ $\comodf(L'),$ it image under $\omega_f$ is denoted by $(M,\rho)$. Let $(N,\rho_N)$ be a sub $L$-comodule of $(M,\rho)$. Thus we have commutative diagram
\begin{displaymath}
\xymatrix{ M \ar[r]^{\rho'}\ar[dr]_{\rho} & M \otimes L' \ar[d]^{\id\otimes f}\\
 N\ar@{^{(}->}[u] \ar[dr]^{\rho_{N}}& M \otimes L\\
                              & N \otimes L \ar@{^{(}->}[u].}
\end{displaymath}
To show that $\rho_N$ comes from a coaction of $L'$ on $Y$ amounts to showing that $\rho_N(N)\subset N\otimes f(L').$ The above diagram shows that $\rho_N(N)\subset N\otimes L\cap M\otimes f(L')$. 
If either $M,N, M/N$ or  $L/f(L')$ are flat over $R$, according to Lemma \ref{lem_intersect}, one has equality
\begin{eqnarray}\label{eq8}
  N{\otimes} L \cap M {\otimes} f(L') = N{\otimes} f(L').
\end{eqnarray}  
Thus $\rho_N(N)\subset N\otimes f(L')$, that is, $\rho'$ restricts to a coaction of $L'$ on $N$.

\bigskip

   Conversely, assume that the functor $\omega_f{}^{\rm o}:\comod(L')^{\rm o}\to\comod(L)^{\rm o}$ is fully faithful and its image is closed under taking special subobjects.  By the flatness of  $K$ over $R$, the functor  $\comod_{\rm f}(L')_K\to\comod_{\rm f}(L)_K{}$ is fully faithful and preserves mono- and epimorphisms. 
Therefore, by the equivalence in  \ref{s.216} the functor   $\comod(L'\otimes K)\to\comod(L\otimes K)$ satisfies conditions  of \cite[Thm~2.21]{DM}, hence $L'\otimes K\to L\otimes K$ is injective. Consequently the map $f:L'\to L$ is injective, (i) is proved. 

Now assume that the image of $\omega_f{}^{\rm o}$  is closed under taking subobjects. The discussion above shows that $f$ is injective. We will identify $L'$ with a subcoalgebra of $L$. For a non-invertible $0\neq\pi\in R$,
denote $C:=\pi L\cap L' \supset \pi  L'$. By means of Lemma \ref{lem_intersect} we see that $C$ is also an $L$-subcomodule of $L'$. Since an $L$-comodule is the union of its finite subcomodule, the assumption on $\omega_f{}^{\rm o}$ implies that $C$ is in fact an $L'$-subcomodule of $L'$: 
$$\Delta (C)\subset C\otimes L'\subset \pi L\otimes L'.$$
Thus for any $c\in C$ we have
$$\Delta (c)=\sum \pi a_i\otimes b_i,\quad a_i\in L,\ b_i\in L'.$$
Hence $c=\sum_i \pi\varepsilon(a_i)b_i\in \pi L'$. That is $C\subset \pi L'$, consequently $\pi L\cap L'=\pi L'$.
The last equation holds for any $\pi\in R$, it follows  that $L/ L' $ is torsion free over $R$, hence flat, as $R$ is a Dedekind ring.
\end{proof}

\begin{remarks}\label{rmk.3.7}
The proof of Proposition \ref{hom_inj}  
 is based on the fact that, when $R$ is a field, the map 
$\coend (\omega | \langle X \rangle)\to L$ is injective. The reader is referred to \cite[Thm~6.4.4]{Sza} for the detailed proof of this fact or to \cite[Lem~1.2]{phh_2016.jaa} for a short proof. In Saavedra's proof of this fact \cite[II.2.6.2.1]{SR72}, the proof of implication d) $\Rightarrow$ a) is incomplete.
\end{remarks}

We now express the local finiteness in terms of Tannakian duality. First we have the following characterization of finite coalgebras.  
\begin{proposition}\label{pro_finite}
Let $L$ be a flat coalgebra over a Dedekind ring $R$. Then $L$ is finite over $R$ if and only if $\comodf(L)$ has a projective generator which is $R$-projective.
\end{proposition}
\begin{proof} Assume that $L$ is finite. Then $L$ is projective over $R$ and there is an equivalence between 
$\comod(L)$ and $\mod(L^\vee)$, where $L^\vee$ is the dual $R$-module to $L$.
%
Now $L^\vee$ is a projective generator in $\comodf(L)$.

Conversely, assume that $\comodf(L)$ has a projective generator, say $P$, which is $R$-projective. Then this category is equivalent to $\modf(A)$, where $A={\sf End}^L(P)$. 
As ${\sf End}_R(P)$ is $R$-finite projective and
  $A$ is an $R$-submodule of ${\sf End}_R(P)$, it is also projective. Therefore $A^\vee$ is a finite flat $R$-coalgebra and   $\mod(A)$ is equivalent to  $\comod(A^\vee)$. By Tannakian duality we conclude that $A^\vee\cong L$, whence $L$ is finite.
\end{proof}

Recall that for a comodule $X$ of $L$, $\langle X\rangle$ denotes the full subcategory of $\comod_f(L)$ consisting of subquotients of finite direct sums of copies of $X$. If $X$ is $R$-finite projective then $\langle X\rangle$ satisfies the conditions of Definition \ref{dominating} (since a submodule of a finite projective $R$-module is again finite projective, we can take for $\langle X\rangle^o$ the full subcategory of subobjects of finite direct sums of copies of $X$). 
Hence, by means of Theorem \ref{thSa}, we obtain a coalgebra 
$\coend (\omega | \langle X \rangle)$, the comodule category of which is equivalent to $\langle X\rangle$.
\begin{proposition}\label{coend_X} Let $L$ be a flat coalgebra over a Dedekind ring $R$ and let $\omega$ denote the forgetful functor $\comodf(L)\to\mod(R)$.
\begin{enumerate}
 \item For each $R$-finite projective comodule $X$ of $L$ the
    natural map $\coend (\omega | \langle X \rangle)\to \coend (\omega)= L$ is injective and special. If $L$ is specially locally finite then  $\coend (\omega | \langle X \rangle)$ is finite over $R$.
\item If for any $R$-finite projective comodule $X$, $\coend (\omega | \langle X \rangle)$ is finite, then $L$ is specially locally finite.
 \end{enumerate}
\end{proposition}
\begin{proof}

(i) The first claim follows from Tannakian duality Theorem \ref{thSa} and Proposition \ref{hom_inj}.

 Assume now that $L$ is specially locally finite. Let $M$ be an $R$-finite projective comodule of $L$. Then by assumption on $L$, there exists a special subcoalgebra $C$ of $L$, which is finite over $R$ and such that the coaction of $L$ on $M$ factors though that of $C$. According Proposition \ref{hom_inj} (ii), $\comod_f(C)$ is a full subcategory of $\comod_f(L)$ closed under taking subobjects. Since 
$M\in \comodf (C)$, we have $\langle M\rangle\subset\comodf(C)$. Hence we have a factorization
$$\coend(\omega|\langle M\rangle)\to C\to L.$$ 
But $C$ is  $R$-finite, whence the claim.

(ii) According to (i), we can cover $L$ by its finite special subcoalgebras  $L = \displaystyle\bigcup_i  L_i,$ where $\{L_i\}$ is a cofinal directed system, i.e. any two coalgebras $L_i,L_j$ are contained in some $L_k$. 
Hence, if $C$ is a finite subcoalgebra of $L$, then $C$ is contained in some $L_i$. Since $L_i$ is saturated in $L$,  $C^{\rm sat}\subset L_i$, hence is finite.
  \end{proof}

Finally we provide a condition for the surjectivity of a colagebra homomorphism.
\begin{proposition}\label{hom_sur}
Let $f:L'\to L$ be a morphism of flat coalgebras over $R$. Then $f$ is surjective if and only if the induced functor $\omega_f{}^{\rm o}:\comodo(L')\to \comodo(L)$ satisfies the following condition: each $M\in\comodo(L)$ is a special subquotient of $\omega_f(N)$ for some $N\in\comodo(L')$.
\end{proposition}
\begin{proof}
Assume that $f$ is surjective. Let $M$ be an  $R$-finite projective $L$-comodule. Then $M$ is a special subcomodule of   $L^{\oplus  r}$ for some $r>0$ (cf. \ref{s.104}). We identify $M$ with a subcomodule of $L^{\oplus  r}$ and choose a generating set $\{m_i\}$ of $M$. Let $m_i'\in L^{'\oplus r}$ be such that $f(m_i')=m_i$. The set $\{m_i'\}$ is contained in some finite module $N$ of $L^{'\oplus r}$. The homomorphic image of $N$ in $L^{\oplus r}$ is denoted by $N_1$, this is an $L$-subcomodule of $L^{\oplus r}$ which contains $M$. Since $M$ is special in $L^{\oplus r}$, it is also special in $N_1$. Thus $M$ is a special subcomodule of the quotient $N_1$ of $\omega_f(N)$, where $N$ is an $L'$-comodule.

Conversely, assume that $\omega_f{}^{\rm o}$ has the stated property. It suffices to show that any $R$-finite subcoalgebra $C$ of $L$ is in the image of $f$. Consider such a $C$ as a (right) $L$-comodule. 
By assumption, there exists an $R$-finite projective $L'$-comodule $N$, such that $C$ is a special subquotient of  $\omega_f(N)$ in $\comodf(L)$. 
Since $C={\sf Cf}(C)$, and $C$ is a special subquotient of $N$, according to subsection \ref{s.104} we conclude that $C\subset {\sf Cf}(N)$. On the other hand, it follows from the construction that ${\sf Cf}(N)$ is the image of ${\sf Cf}_{L'}(N)$ - the coefficient space of $N$ considered as an $L'$-comodule. Thus the map $L'\to L$ is surjective.
\end{proof}
 
\section{Tannakian description of group scheme homomorphisms}\label{Sect4}
Let $R$ be a Dedekind ring with fraction field $K$.
Let $f:G\to G'$ be a homomorphism of flat affine group schemes over $R$. 
We say that $f$ is surjective or a quotient homomorphism if it is faithfully flat. In the first part of this chapter we give a necessary and sufficient condition for the faithful flatness of $f$, in terms of Tannakian duality. Then we will give a condition for $f$ to be a closed immersion. Finally we give a criterion for the exactness of a sequence of group homomorphisms.

\subsection{Faithfully flat homomorphisms and closed embedding}  
The following theorem is a generalization of the well-known faithful flatness theorem for Hopf algebras: {\em a (commutative) Hopf algebra is faithfully flat over any Hopf subalgebra}, see, e.g.,\cite[Thm~14.1]{Wat}. The proof is developed from an idea of J.C. Moore \cite{Moore}. The advantage of this proof is that we don't need to assume the Hopf algebras involved to be of finite type.
\begin{theorem}\label{flatness}
  Let $L$  be a flat commutative Hopf algebras over $R$ and $L'$ be a Hopf  subalgebras. Then  $L$ is faithfully flat over $L'$  if and only if  $L/L'$ is $R$-flat, i.e., $L'$ is a special subcoalgebra of $L$.
\end{theorem}
\begin{proof} 
Assume that  $L$ is faithfully flat over $L'$.  Consider the tensor product of the exact sequence
\begin{eqnarray}
\label{eq6}
\xymatrix{0\ar[r]&L' \ar[r] &L\ar[r] &\ar[r]Q&0 }
\end{eqnarray}
  with $L$ over $L'$ we get an exact sequence
\begin{eqnarray}\label{eq7}
\xymatrix{0\ar[r]&L\ar[r]& L {\otimes}_{L'} L\ar[r]&L{\otimes}_{L'}Q\ar[r]& 0}
\end{eqnarray}
The multiplication $L{\otimes}_{L'} L\longrightarrow L$ splits this sequence, hence
$L{\otimes}_{L'}Q$ is flat over $L.$ Since $L$ is faithfully flat over $L',$ $Q$ is flat over $L'$ and therefore it is flat over $R.$

Conversely, assume that $L/L'$ is $R$-flat.
We first show that $L$ is flat over $L'$, i.e., for any $L'$-module $M$, ${\rm Tor}^{L'}_1(M,L)=0$.
The claim holds if $R$ is a field.

First assume that $M$ is $R$-flat. Choose  $P_{*}$, a projective resolution of $L$ as an $L'$-module. Then $P_{*} \otimes k$ is a projective resolution of $L \otimes k$ over $L' \otimes k$, where $k$ is a residue field or the fraction field of $R$. We have
$$ (M \otimes k)  \otimes_{(L' \otimes k)}(P_{*} \otimes k) \cong (M \otimes_{L'} P_{*}) \otimes k, $$ implying
$$H_i ((M \otimes_{L'} P_{*}) \otimes k) \cong {\rm Tor}^{L' \otimes k}_i (M \otimes k, L \otimes k),\mbox{ for all } i \ge 0.$$
Since $M \otimes_{L'} P_{*}$ is flat over $R$, a Dedekind ring, the universal  coefficient theorem, (see, e.g., \cite[Thm~3.1.6]{Weibel}), applies. Thus, for each $i\geq 1$, we have an exact sequence
$$ 0\to H_i (M \otimes_{L'} P_{*}) \otimes k \to H_i ((M \otimes_{L'} P_{*}) \otimes k) \to {\rm Tor}^{R}_1({ H_{i-1}}(M \otimes_{L'} P_{*}), k)\to 0.$$
That is, for all $i\geq 1$,
$$0\to {\rm Tor}^{L'}_i (M, L) \otimes k \to {\rm Tor}^{L' \otimes k}_i (M \otimes k, L \otimes k) \to {\rm Tor}^{R}_1({\rm Tor}^{L'}_{i-1} (M, L) , k)\to 0.$$
As $L/L'$ is flat, the map
$L' \otimes k \longrightarrow L\otimes k$ is injective, hence  flat. 
 Therefore
${\rm Tor}^{L' \otimes k}_i (M \otimes k, L \otimes k) = 0$, for all $i\geq 1$. Consequently
$${\rm Tor}^{R}_1({\rm Tor}^{L'}_{i-1} (M, L) , k) = {\rm Tor}^{L'}_i (M, L) \otimes k  =0, \mbox{ for all } i\geq 1.$$
This holds for any residue field and the fraction field of $R$, hence ${\rm Tor}^{L'}_0 (M, L)$ is  flat over $R$  and 
${\rm Tor}^{L'}_i (M, L) =0$ for all $i\geq 1$.

Let now $M$ be an arbitrary  $L'$-module. Then the $R$-torsion submodule $M_\tau$ of $M$ is also an $L'$-submodule. The quotient module $M/M_\tau$ is then $R$-flat.  As we have the exact sequence
$${\rm Tor}^{L'}_1(M_\tau,L)\to {\rm Tor}^{L'}_1(M,L)\to {\rm Tor}^{L'}_1(M/M_\tau,L)\to\ldots$$
it suffices to show ${\rm Tor}^{L'}_1(M,L)=0$ for $M$ being $R$-torsion.

For each non-zero ideal $p\subset R$, the submodule $M_p$ of elements annihilated by $p$, is also an $L'$-submodule. As $M$ is torsion, it is the direct limit of $M_p$. Since the Tor-functor commutes with direct limits, one can replace $M$ by some $M_p$. Since $R$ is a Dedekind ring, each non-zero ideal $p$ is a product of finitely many prime ideals. Therefore each $M_p$ has a filtration, each grade module of which is annihilated by a certain non-zero prime ideal. Thus using induction we can reduce to the case $M$ is annihilated by a prime ideal $p$. In this case $M=M\otimes k_p$ is an $L'\otimes k_p$-module, where $k_p:=R/p$ and we have
$$M\otimes_{L'}P_*= M\otimes_{L'\otimes k_p}(P_*\otimes k_p).$$
 Since $P_*\otimes_Rk_p$ is an $L'\otimes_Rk_p$-projective resolution of $L\otimes_Rk_p$, we see that
 $${\rm Tor}^{L'}_i(M,L)={\rm Tor}^{L'\otimes k_p}_i(M,L\otimes k_p)=0,$$
 as $L\otimes k_p$ is flat over $L'\otimes k_p$.
 
 Finally, we show that $L$ is faithfully flat over $L'$.
 
 Let $M$ be an $L'$-module, such that $M\otimes _{L'}L=0$. Then we have
$$M_k\otimes _{L'_k}L_k\cong (M\otimes _{L'}L)\otimes _Rk=0$$
 for any residue field $k$ or the fraction field $K$ of $R$.
Since  $L'_k\to L_k$ is faithfully flat, we have $M_k=0$ and $M_K=0$. If $M$ is finite over $L'$, this implies that $M=0$, according to \cite[Thm~4.9]{Mats}. In the general case, $M$ always contains a non-zero finite submodule and since $L$ is flat over $L'$, we see that $M=0$ if $M\otimes_{L'}L=0$. Thus $L$ is faithfully flat over $L'$.  
\end{proof}
  
  As a corollary of Theorem \ref{flatness} and Propositions \ref{hom_inj}, \ref{hom_sur} we have the following theorem.
 \begin{theorem} \label{th1}
Let $f: G  \longrightarrow G'$ be a homomorphism of  affine flat groups over $R,$ and $\omega_f^{\rm o}$  be the corresponding functor
 $\repo(G') \longrightarrow \repo(G).$ 
 \begin{enumerate}
 \item $f$ is faithfully flat if and only if $\omega_f{}^{\rm o}: \repo(G') \longrightarrow \repo(G)$ is fully faithful and its image is closed under taking subobjects. 
 \item  $f$ is a closed immersion if and only if every object of  $\repo(G)$ is isomorphic to a special subquotient of
 an object of the form  $\omega_f(X'),$ $X' \in \repo(G').$
 \end{enumerate}
\end{theorem}

\begin{remarks}\label{rmk_co_ex}
 (i) For (commutative) Hopf algebras over a field, any injective homomorphism $L'\to L$ is automatically faithfully flat. This is not the case for Hopf algebras over a Dedekind ring. Take for example $R=k[x]$ and $L=R[T]=k[x,T]$ with the $R$-coalgebra structure given by $\Delta(T)=T\otimes 1+1\otimes T$, $\varepsilon(T)=0$  $S(T)=-T$ (i.e. $L=R[\mathbf G_a]$), and consider the map of Hopf $R$-algebras
$$f:L\to L, \quad T\mapsto xT.$$
The quotient $L/f(L)$ is not flat over $R$, hence Theorem \ref{th1} implies that $L$ is not faithfully flat over $f(L)$.

(ii) Theorem \ref{th1} implies that, if $f:L'\to L$ is a homomorphism of flat Hopf $R$-algebras, such that $L/f(L')$ is flat over $R$, then $L$ is faithfully flat over $f(L')$.

(iii) The claim (ii) of Theorem \ref{th1} generalizes a result of dos Santos, \cite[Prop.~12]{JSa}. 
\end{remarks}

\begin{questions}\label{Q2} Theorem \ref{th1} suggests the following question: find a criterion of a Tannakian category $\mathcal C$ such that its Tannakian group $G$ is pro-algebraic in the sense that
$$G=\varprojlim_\alpha G_\alpha$$
where each group scheme $G_\alpha$ is of finite type and each structure map $G_\beta\to G_\alpha$ is faithfully flat. Similarly, find a criterion such that $G$ is smooth in the sense that the $G_\alpha$ above are smooth. It seems that the local finiteness mentioned in the previous section closely relates to these problems.\end{questions}

\subsection{Tannakian description of exact sequences}
Recall that the right regular representation of $G$ on its coordinate ring $L$, $(g,h)\mapsto gh: (gh)(x)=h(xg)$, $g\in G$, $h\in L$ corresponds to the right coaction of $L$ on itself by the coproduct $\Delta$. The left regular action of $G$ on $L$, $(g,h)\mapsto gh: (gh)(x)=h(g^{-1}x)$ corresponds to the following  (right) coation of $L$ on itself:
$$a\mapsto \sum_ia'_i\otimes S(a_i),\quad\text{  where  }\Delta(a)=\sum_ia_i\otimes a_i', \text{ and $S$ denotes the antipode}.$$

Let $ G \longrightarrow A$ be a homomorphism of affine groups schemes over $R$. Let
$I_A$ be the kernel of counit $\epsilon: R[A] \longrightarrow R$, i.e. the augmentation ideal of $R[A]$, and let  $I_AR[G]$ be the ideal generated by the image of $I_A$ in $R[G].$ Then the kernel of $ G \longrightarrow A$ is the closed subscheme of $G$ with coordinate ring $R[G]/I_AR[G].$ 
A sequence 
\begin{displaymath}
\xymatrix{1 \ar[r] &  H \ar[r]^{q}&G\ar[r]^{p}& A \ar[r] &1}    
 \end{displaymath}                  
is said to be exact if  $p$  is a quotient map with kernel $H.$ We will provide a criterion for the exactness in terms of the functors 
\begin{eqnarray}\label{eq}
\xymatrix{\repo(A)\ar[r]^{p^*}&\repo(G) \ar[r]^{q^*}& \repo(H).}
\end{eqnarray}

We first need a lemma relating the coordinate rings $R[A]$, $R[G]$ and $R[H]$. 
\begin{lemma}\label{l2}
Let $ G \longrightarrow A$ be a quotient map with kernel $H.$  \begin{itemize}
 \item[(i)] If $M$ is a $G$-module, then $M^H$, the $H$-trivial submodule of $M$, is stable under the action of $G$, i.e. it is a $G$-submodule of $M.$
 \item [(ii)]   $R[A] $ is  equal to $ R[G]^H$ as $G$-modules.
\end{itemize}
\end{lemma}

\begin{proof} (i) This follows immediately from the normality of $A$ in $G$, see \cite[I.3.2]{JC} for a proof.

(ii) The proof  is  based on the fact that $R[G]$ is faithfully flat over $R[A]$ and follows closely the proof for group schemes over fields, cf. \cite[Sect.~15.4]{Wat}. 

There is an isomorphism 
$H\times G\xrightarrow{\simeq} G\times_AG ;\quad (h,g)\to (hg,g),$
which precisely means that $G \longrightarrow A$ is a principal bundle under $H.$ In terms of the coordinate rings this isomorphism has the form
\begin{equation}\label{eq_PB}
\varphi:  R[G] \otimes_{R[A]} R[G] \cong  R[H] \otimes_R R[G],\quad a\otimes b\mapsto \sum_iq(a_i)\otimes a'_ib,
\end{equation}  
where $q:R[G]\to R[H]$ denotes the quotient map and $\Delta(a)=\sum_ia_i\otimes a'_i$. The inverse map is given by
$$q(a)\otimes b\mapsto \sum_ia_i\otimes S(a_i')b,$$
where $S$ denotes the antipode of $R[G]$. One checks that this assignment depends on $q(a)\in R[H]$ but does not depend on the choice of $a\in R[G]$. Now consider the following diagram:
$$\xymatrix{R[G]^H\ar[r]&R[G]\ar[r]^{d\qquad}\ar[rd]_{d'}& R[G]\otimes_{R[A]}R[G]\ar[d]^\varphi\\
&&R[H]\otimes_RR[G]}$$
where $d(a):=a\otimes 1-1\otimes a$. Then $d'$ is computed as follows:
$$d'(a)=\sum_iq(a_i)\otimes a'_i-1\otimes a,\quad{\rm where }\quad \Delta(a)=\sum_ia_i\otimes a'_i.$$
Thus $R[G]^H$ is precisely the kernel of $d'$, hence is also the kernel of $d$ as $\varphi$ is an isomorphism. On the other hand, as $R[G]$ is faithfully flat over $R[A]$, $R[A]$ is the kernel of $d$. We conclude that $R[A]=R[G]^H$.
\end{proof}

\begin{theorem}\label{th3}
  Let us be given a  sequence
 \begin{displaymath}
 \xymatrix{H \ar[r]^{q}&G\ar[r]^{p}& A}
 \end{displaymath}
 with $q$ a  closed  immersion  and $p$ faithfully  flat. 
 Then this sequence is  exact  if  and  only  if  the  following  conditions  are  fulfilled: \begin{itemize}
  \item[(a)]  For  an  object  $V \in \repo(G),$  $q^*(V )$  in  $\repo(H)$ is  trivial  if  and  only  if 
            $V \cong p^*U$  for some  $U \in \repo(A).$ 
 \item[(b)]  Let $W_0$   be the maximal trivial subobject of $q^*(V )$ in $\repo(H).$ Then there 
            exists $V_0 \subset V \in \repo(G),$ such that $q^*(V_0) \cong W_0.$
 \item[(c)]  Any $W \in \repo(H)$  is a quotient in (hence,  by taking duals,  a subobject 
            of) $q^*(V )$ for some  $V \in \repo(G).$                             
\end{itemize}
\end{theorem}
\begin{proof}
Assume that $q : H \longrightarrow G$ is the kernel of $p  : G  \longrightarrow A.$ Then (a) and (b) follow from  \ref{l2} (i) and (ii). We prove (c). 

Let ${\sf Ind} : {\sf Rep}(H) \longrightarrow {\sf Rep}(G)$ be the induced representation functor, it is the 
right  adjoint  functor  to  the  restriction  functor  ${\sf Res}  :  {\sf Rep}(G)  \longrightarrow  {\sf Rep}(H)$ that  is
\begin{eqnarray} \label{eqiso}
\xymatrix{{\sf Hom}_G (V, {\sf Ind}(W ))  \ar[r]^{\cong}& {\sf Hom}_H ({\sf Res}(V ), W).}
\end{eqnarray}      
 One has $ {\sf Ind}(W ) \cong ( W \otimes_R R[G] )^H$ , where  $H$ acts on $R[G]$ by the left regular action, \cite[I.3.4]{JC}.  
Notice that the subspace $ ( W \otimes_R R[G] )^H\subset W\otimes_RR[G]$ is invariant under the action of $R[A]$, i.e. it is an $R[A]$ submodule.

Notice that the isomorphism \eqref{eq_PB}
$$
  R[G] \otimes_{R[A]} R[G] \xrightarrow\simeq  R[H] \otimes_R R[G]
$$
is a map of $G$-modules where $G$ acts on the second tensor terms by the right regular action.
Since $R[G]$ is faithfully flat over 
its subalgebra $R[A]$, taking the tensor product with $R[G]$ over $R[A]$ commutes with taking $H$-invariants,  hence we have  
$${\sf Ind}(W ) \otimes_{R[A]} R[G] \simeq
( W \otimes_R R[G] )^H\otimes_{R[A]}R[G]\simeq W \otimes_R R[G]. $$
 This in turns implies that the functor ${\sf Ind}$ is faithfully exact.

Setting $V = {\sf Ind}(W ) $ in \eqref{eqiso}, one obtains a canonical map
 $u_W : {\sf Ind}(W ) \longrightarrow
W $  in  ${\sf Rep}(H)$ which gives back the isomorphism in \eqref{eqiso} as follows: 
$${\sf Hom}_G (V, {\sf Ind}(W )) \ni  h \mapsto u_W \circ  h  \in {\sf Hom}_H ({\sf Res}(V ), W).$$            
The map $u_W$   is  non-zero whenever  $W$  is non-zero. Indeed,  since ${\sf Ind}$ is exact and faithful, ${\sf Ind(W)}$ is  non-zero whenever  $W$ is  non-zero. Thus if were $u_W= 0,$  then \eqref{eqiso} were the zero map for any $V.$ On the other hand, for 
$V = {\sf Ind}(W ), $ the right 
hand side contains the identity map. A contradiction, which  shows that $u_W$ cannot vanish.

We  show now that $u_W$ is  always  surjective.  
Let  $U  =  {\sf im} (u_W ) $ and $T = W/U \in \repf(H).$ We have the following diagram 
\begin{displaymath}
\xymatrix {
0\ar[r]&{\sf Ind}(U) \ar[d]^{u_U} \ar[r] &{\sf Ind}(W ) 
\ar[d]^{u_W }\ar[r] &  {\sf Ind}(T ) \ar[r] \ar[d]^{ u_T  }&0 \\
0\ar[r]&U \ar[r]& W \ar[r] &T \ar[r]&0 \\}
\end{displaymath}
By assumption, the composition  ${\sf Ind}(W ) \twoheadrightarrow {\sf Ind}(T ) \longrightarrow T$  is $0,$ therefore $ {\sf Ind}(T ) \longrightarrow
T$ is a zero map, implying $T  = 0.$ 

  Assume now $W$ is finitely generated projective over $R$.   Then ${\sf Ind}(W )$ is torsion free. Hence ${\sf Ind}(W )$ is the union of its finitely generated $R$-projective modules, we can find a finitely generated $G-$submodule $W_0 (W)$ of ${\sf Ind}(W )$ which still maps 
surjectively on  $W.$ In order to  obtain the statement on the embedding  of $W,$ we 
dualize the map $W_0 (W^\vee) \twoheadrightarrow W^\vee. $

Assume now that (a), (b), (c) are satisfied. Then it follows 
from  (a) 
that  for  $U \in \repo(A)$,  $q^* p^*(U)  \in \repo(H) $ is  trivial.  Hence  $ pq :  H \longrightarrow A$  is  the 
trivial  homomorphism.  Recall  that  by  assumption,  $q$ is  injective, $p$ is  surjective.    
Let $ \overline {q} : \overline {H} \longrightarrow G $ be the kernel of $p.$ Then we have commutative diagram 

\begin{displaymath}
\xymatrix{ H\ar[dr]_{q}\ar[rr]^{i} &&\overline {H} \ar[dl]^{\overline {q}} \\
&G}\longleftrightarrow 
\xymatrix{ \repo(H)  &&\repo(\overline {H}) \ar[ll]_{i^*}  \\
&\repo(G) \ar[ul]^{q^*}\ar[ur]_{\overline {q}^*}.}
\end{displaymath}  

It  remains  to show  that $i$ is surjective. We use Proposition \ref{hom_inj}.
    We first show the functor $i^*$ is fully faithful. The faithfulness is obvious, we show the fullness. Let $\overline W_0,\overline W_1$ be objects in $\repo(\overline H)$, and $\varphi: W_0=i^*(\overline W_0)\to W_1=i^*(\overline W_1)$ be  a morphism of $H$-modules.
Since $\overline H$ is the kernel of $p$, the first part of this proof shows that there exist surjection $\overline q^*(V_0)\surj \overline W_0$ and injection $\overline W_1\inj\overline q^* (V_1)$ where $V_0,V_1$ are objects of $\repo(G)$. Thus $\varphi$ combined with these maps yields a map $\widehat\varphi:q^*(V_0)\to q^*(V_1)$.
$\widehat\varphi$ corresponds to an element of $(q^*(V_1)\otimes q^*(V_0)^\vee)^H$.
  By conditions (a), (b) for $H$ and by the fact that $\overline H$ also satisfies (a), (b), there exists $\psi:\overline q^*(V_0)\to \overline q^*(V_1)$ such that
$$\widehat\varphi=i^*(\psi):
 q^*(V_0)\surj\overline W_0\xrightarrow{\varphi}
\overline W_1\inj q^* (V_1).
$$
This implies that $\varphi=i^*(\overline \varphi)$ for some $\overline \varphi:\overline W_0\to\overline W_1$. Thus $i^*$ is full.

For any $W \in \repo(H),$   by (c) there exist $V_0, V_1$  in $\repo(G)$ 
and  $ \varphi :  q^*(V_0  ) \longrightarrow q^*(V_1)$ such  that  $W= {\sf im} \varphi. $ Since $ i^*$ is full, $\varphi = i^*\overline {\varphi},$ hence  $W \cong i^*({\sf im} \overline {\varphi}).$ Thus  we  have proved that  any  object  in $\repo(H)$ is isomorphic to the image under $i^*$ of an object in $\repo(\overline {H}).$ Together with 
the discussion above this implies that $ \overline {H} \cong H.$          
\end{proof}

\section{Stratified sheaves on a smooth scheme over a Dedekind ring}\label{Sect5}

\subsection{Stratified sheaves on smooth schemes over a Dedekind ring}\label{Sect4.2}

 Let $R$ be  a Dedekind ring, denote $S:={\sf Spec}(R)$. We shall assume that the residue fields of $R$ are all {\em perfect}.
 Let $f:X\to S$ be a smooth morphism with geometrically connected fibers. 
Consider the category $\str(X/S)$ of $\mathcal O_X$-coherent modules over the sheaf $\Diff(X/S)$ of algebras of differential operators on $X/S$. 
This is an abelian tensor category with the unit object being $\mathcal O_X$ equipped with the usual K\"ahler differential. 
We call objects of $\str(X/S)$ stratified sheaves. 

If $R$ is a perfect field then objects of $\str(X/S)$ are locally free as sheaves on $X$ (see \ref{str/k}). 
For the proof of this fact in characteristic 0 see \cite{Katz} and in positive characteristic see \cite{JSa0}.

We assume that $f$ admits a section $\xi:S\to X$.
The pull-back along $\xi$ provides a functor $\xi^*:\str(X/S)\to {\sf Mod}(R)$. The following results are similar to those of dos Santos \cite[Sect~4.2]{JSa}. 
\begin{proposition}\label{pro_santos}
The following claims hold:
\begin{itemize}
\item[(i)] An object of $\str(X/S)$, which is $R$-torsion-free, is flat (and hence is locally free) as an $\mathcal O_X$-module. Consequently, the subcategory of $\mathcal O_X$-locally free objects is closed under taking subobjects.
\item[(ii)] The functor $\xi^*:\str(X/S)\to{\sf Mod}(R)$ is faithful and exact.
\end{itemize}
\end{proposition}
\begin{proof} Since these are local properties, we can assume that $R$ is a discrete valuation ring with a uniformizer $t$ and $X$ is an affine scheme over $R$, $X={\sf Spec} A$. Then the proof of \cite[Lem.~19, Cor.~20]{JSa} can be used. 

(i).  By the local flatness criterion, an object $M$ of $\str(X/S)$ is flat over $\mathcal O_X$ if ${\sf Tor}_1^R(M,\mathcal O_{X_0})=0$ and $M_0:=M/tM$ is flat over $X_0=X\times_{{\sf Spec}(R)}{\sf Spec}(R/tR)$. The second condition is trivially satisfied as $X_0$ is a scheme over a field $R/tR$. The first condition just means $M$ is $t$-torsion free.

(ii) We show that $\xi^*$ is left exact.  Let $m$ be the kernel of $\xi:A\to R$ then $\xi^*$ is the functor tensoring with $A/m=R$. 
Shrinking $X$ if necessarily, we can assume the existence of a regular sequence of generators of $m$, say $x_1,x_2,\ldots, x_n$. It suffices to check that $${\sf Tor}^A_1(M,A/m)$$ vanishes for any $M\in \str(X/S)$. 
Let $M_\tau$ be the $R$-torsion part of $M$ and $M_f:=M/M_\tau$ then $M_f$ is $R$-torsion free hence is $\mathcal O_X$-flat by (i). Thus one is led to check the claim for those $M$, which are $R$-torsion. Such an $M$ has a filtration (which is finite as $M$ is coherent)
$$0=M_0\subset M_1\subset\ldots\subset M_m=M,$$
where $M_i/M_{i-1}$ is killed by  $t$, so that  $M_i/M_{i-1}$ is supported on $X_0$. 
Using the long exact sequence for Tor, one reduces the problem to the case $M$ is supported on $X_0$. Thus $M$ is locally free as a sheaf on $X_0$.
Let $T_j:={\sf Tor}^A_1(M,A/(x_1,\ldots,x_j))$. We will show by induction that $T_j=0$ for $j=1,2,\ldots,n$. To see that $T_1=0$ we consider the exact sequence
$$0\to A\xrightarrow{\cdot x_1}A\to A/(x_1)\to 0,$$
which shows that $T_1$ is the kernel of the multiplication by $x_1$ on $M$, which is 0 as $M$ is locally free over $\mathcal O_{X_0}$. For the induction step, assuming that $T_{j-1}=0$ and considering the short exact sequence
$$0\to A_{j-1}\xrightarrow{\cdot x_j}\to A_{j-1}\to A_j\to 0,$$
$A_j:=A/(x_1,x_2,\ldots,x_j)$ we obtain a short exact sequence
$$0\to T_j\to A_{j-1}\otimes_AM\xrightarrow{\cdot x_j}\to A_{j-1}\otimes_AM.$$
Now $A_{j-1}\otimes_AM$ is a stratified module on the smooth scheme
$${\rm Spec } A/(t,x_1,\ldots,x_{j-1}),$$
hence, as above, we conclude that $T_j=0$.

Finally, to see that $\omega$ is faithful, it suffices to see that $\xi^*(M)=0$ implies $M=0$ for any stratified sheaf on $X$. We follow again the argument above: the torsion-free part is restricted to the generic fiber and  the torsion part is filtered with sucessive quotients supported on closed fiber.\end{proof}

A locally free stratified sheaf is also called {\em stratified bundle}.
 Let $\str^\text{o}(X/S)$ be the subcategory of stratified bundles. This is a rigid tensor category, and by the above proposition it is closed under taking subobjects. 

\begin{lemma}Let $K$ be quotient field of $R$. Then the category  obtained by scalar extension $R\to K$, $\str^{\rm o}(X/S)_K$, is abelian.\end{lemma}
\begin{proof}
It is to define quotients in $\str^{\rm o}(X/S)_K$. We first show the following: for $M\subset N$ in $\str^{\rm o}(X/S)$, let $M^{\rm sat}$ denote the saturation of $M$ in $N$, i.e. the minimal extension of $M$ in $N$ such that $N/M^{\rm sat}$ is $R$-torsion free, then $M$ and $M^{\rm sat}$ are isomorphic as objects in $\str^{\rm o}(X/S)_K$.  
Indeed, since $M^{\rm sat}/M$ is $R$-torsion, as in the proof of Proposition \ref{pro_santos} (ii), there exists a finite filtration for $M^{\rm sat}/M$ such that each successive quotient is killed by a non-zero element of $R$. Hence $M^{\rm sat}/M$ is killed by a single non-zero element $a\in R$. Thus $M^{\rm sat}/M\otimes_RK=0$, that is, $M\otimes_RK\cong M^{\rm sat}\otimes_RK$.

Let now $f:M\to N$ be a morphism in $\str^{\rm o}(X/S)_K$, then there exists $a\in R$ such that $af$ is a morphism in $\str^{\rm o}(X/S)$. The kernel of $f$ is defined to be the kernel of $af$, since $M$ is $R$-torsion free, the definition is independent of the choice of $a$. On the other hand, the co-kernel of $f$ is defined to be the quotient of $N$ by the saturation of ${\sf Im}(af)$. The discussion above shows that this definition is independent of the choice of $a$. Indeed, the saturations of  ${\sf Im}(af)$ and  ${\sf Im}(baf)$ in $N$ are the same for any $b\in R$.
\end{proof}

Based on results of section \ref{Sect2} we make the following definition.
\begin{definition}The fiber functor $\xi^*$ makes $\str^{\rm o}(X/R)$ a Tannakian lattice, its Tannakian group, denoted by $\pi(X/S,\xi)$, is called the relative fundamental group scheme of $X/S$.\end{definition}

Let $M\in\str^\text{o}(X/S)$. Consider the full subcategory $\langle M\rangle_s^\otimes$ of $\str^\text{o}(X/S)$, consisting of special subquotients of direct sums of tensors powers of the form
$$T^{a,b}(M):=M^{\otimes a}\otimes M^{\vee\otimes b},\quad a,b\in\mathbb N.$$
Then $\langle M\rangle_s^\otimes$ is also a Tannakian lattice.  Its Tannakian group scheme $G(M)$ is called the differential Galois group scheme of $M$. This group was first studied in \cite{JSa}.

\subsubsection{}  
We don't know if $\str(X/S)$ is a Tannakian category over $R$, be cause we cannot check if any stratified sheaf is representable as a quotient of a stratified bundle. However we can define the abelian envelope  of $\str^{\rm o}(X/S)$, denoted by
 $\mathcal C=\mathcal C(X/S)$,
 as the full subcategory of $\str(X/S)$ consisting of stratified sheaves which can be represented as quotients of stratified bundles. According to Proposition \ref{pro_santos}, $\mathcal C$ is a (neutral) Tannakian category over $R$. We conclude that $\mathcal C$ is equivalent to the representation category of $\pi(X/S,\xi)$ by means of the fiber functor $\xi^*$. Thus $\mathcal C$ is the abelian envelope of $\str^{\rm o}(X/S)$ in the sense of \ref{s.41}.

\subsubsection{}\label{str/k}
Let $s$ be a closed point of $S$, $k:=k_s=R/p_s$ -- the residue field of $s$, and  let $X_s$ denote the fiber of $f$ at $s$. Consider the category $\str(X_s/k)$. Its objects are automatically locally free as $\mathcal O_{X_s}$-modules, cf. \ref{Sect4.2}.
Thus $\str(X_s/k)$ is an abelian rigid tensor category over $k$. 

The fiber of $\str(X/S)$ at $s\in S$ is defined as in Remark \ref{rmk.fiber}, and is denoted by $\str(X/S)_s$. This full subcategory of $\str(X/S)$ is identified with the category of stratified bundle on $X_s$. Indeed, the restriction functor $\str(X/S)\to \str(X_s/k)$ (given by pulling-back along the closed immersion $X_s\to X$) can be identified with the functor which associates to each object $M$ of $\str(X/S)$ the quotient $M/p_sM$. In particular, $\str(X_s/k)$ is naturally a subcategory of $\str(X/S)$, consisting of those objects which are annihilated by $p_s$.

\begin{lemma} $\mathcal C(X/S)_s$ is a full subcategory of $\str(X_s/k)$, closed under taking subquotients.\end{lemma}
\begin{proof}
 This is obvious by definition (see \ref{rmk.fiber}). A stratified sheaf $M_0$ is an object of $\mathcal C(X/S)_s$ if it is annihilated by $p_s$ and can be presented as a quotient of a stratified bundle, say $M$. Thus any quotient of $M_0$ is again in $\mathcal C(X/S)_s$. On the other hand, if $N_0$ is a subobject of $M_0$ then taking the pull-back of $N_0$ along the projection map $M\to M_0$ we get a subobject $N$ of $M$ which surjects onto $N_0$. Since $N$ itself is locally free, we conclude that $N_0$ is in $\mathcal C(X/S)_s$. 
\end{proof}

\begin{remarks} In general we don't know if $M$ can be presented in the from $M_0=M/\pi M$, that is if $M_0$ is the fiber at $s$ of some stratified bundle $M$. The difference between $\mathcal C(X/S)$ and $\str(X/S)$ consists essentially in those stratified sheaves supported in $X_s$ and not representable as a quotient of a stratified bundle. See \ref{complet.dvr} below.\end{remarks}

\subsubsection{} In the rest of this subsection we shall assume that $R$ contains a field $k$ which is mapped isomorphically onto any residue field. Thus $R$ can be see as the coordinated ring on the affine curve $S$ on $k$ (recall that $k$ is assumed to be perfect).

Consider $X$ as a scheme over $k$. A stratified sheaf on $X/k$ is automatically locally free on $X$ and is called an {\em absolute} stratified bundle.
There is a natural ``inflation" functor from $\str(X/k)$ to $\str^{\rm o}(X/S)$: consider an (absolute) stratified bundle on $X/k$ as a (relative) stratified bundle on $X/S$. In the other direction, pulling-back along $f$ yields the functor $\omega_f=f^*:\str(S/k)\to \str(X/k)$. Thus, to summarize, we have the following sequence of tensor functors:
$$\xymatrix{
\str(S/k)\ar[r]^{\omega_f}& \str(X/k)\ar[r]^{{\rm infl}} &\mathcal C(X/S)\ar[r]^{{\rm res}}&\mathcal C(X/S)_s\ar[r]& \str(X_s/k)}.$$

Let $x$ be a $k$-rational point of $X/k$.
The fiber functor at $x$ makes $\str(X/k)$ a neutral Tannakian category. Its Tannakian group is denoted by $\pi(X/k,x)$ and called the fundamental group scheme of $X$ at $x$. Let $s=f(x)$. The functor $\omega_f$ is compatible with the fiber functors at $s$ and $x$. Thus, according to \ref{rmk2.11}, we have a homomorphism of fundamental group schemes $f_*:\pi(X/k,x)\to \pi(S/k,s)$. The restriction functor  $\str(X/k)\to \str(X_s/k)$ is also a tensor functor and is compatible with the fiber functors at $x$, hence induces a homomorphism $\pi(X_s/k,x)\to \pi(X/k,x)$. 

Assume now that $x=\xi(s)$:
$$\xymatrix{
& X\ar[d]_f\\
{\sf Spec}(k)\ar[r]_s\ar[ru]^x&S\ar@/_10pt/[u]_\xi
}$$
 Then the fiber functors $x^*$, $\xi^*$ and $f_s^*$ are compatible in the sense that the following diagram is commutative:
$$\xymatrix{\str(X/k)\ar[r]\ar[rdd]_{x^*}&\mathcal C(X/S)_s\ar[r]\ar[dd]^{\xi^*\otimes_Rk}& \str(X_s/k)\ar[ddl]^{f_s^*}\\
\\
 &{\sf Vect}(k)}$$
Hence it yields a sequence of group scheme homomorphism 
$$\pi(X_s/k,x)\to \pi(X/S,\xi)_s\to \pi(X/k, x)\to \pi(S/k,s).$$

\begin{proposition}  The homomorphism $\pi(X_s/k,x)\to \pi(X/S, \xi)_s$ is surjective.
\end{proposition}
\begin{proof} Here, the group schemes are defined over a field. Hence we can use the criterion for surjectivity of Deligne-Milne \cite[Thm.~2.21]{DM}. We show that for each object $M\in\str(X/S)_s$, when considered as object in $\str(X_s/k,x)$ all its subobjects will be an object in $\str(X/S)_s$. This is obvious from the fact that the category $\str^{\rm o}(X/S)$ is closed under taking subobjects. There exists by assumption an $X\in\str^{\rm o}(X/S)$ which surjects on $M$. We have the following pull-back diagram
$$\xymatrix{X\ar@{->>}[r]&M\\
Y\ar@{^(->}[u]\ar@{->>}[r]&N\ar@{^(->}[u]}$$
Since $Y$ is a subobject of $X$, it is itself locally free.
\end{proof}

\subsubsection{}
In \cite{JSa2} dos Santos proved that the following homotopy sequence is exact: 
\begin{equation}\label{eq_santos}\pi(X_s,x)\to \pi(X,x)\to \pi(S,s)\to 1,\end{equation}
provided that $f$ is a proper map. Hence in this case we also have
an exact sequence
$$  \pi(X/S,\xi)_s\to \pi(X/k, x)\to \pi(S/k,s)\to 1.$$

The general Tannakian duality applied to $\xi^*$ and the categories $\str(X/k)$ and $\str(S/k)$ yields the fundamental groupoid schemes $\Pi(X/k,\xi)$ and $\Pi(S/k,\xi)$, and the functor $\omega_f:\str(S/k)\to \str(X/k)$ yields a surjective homomorphism $f_*:\Pi(X,\xi)\to\Pi(S,\xi)$. The kernel of this groupoid homomorphism, is by definition
$$L:=S\times_{\Pi(S,\xi)}\Pi(X,\xi),$$
where the map $S\to \Pi(X,\xi)$ is given by the unit element.
This is a flat group scheme over $S$. On other hand, the inflation functor $\str(X/k)\to \str(X/S)$ induces a homomorphism $\pi(X/S,\xi)\to \Pi(X/k,\xi)$. The following question is motivated by dos Santos' result mentioned above. 
\begin{questions}   Assume that $f:X\to S$ is a smooth, proper map with connected fibers. Is the following sequence exact 
$$\pi(X/S,\xi)\to \Pi(X/k,\xi)\to \Pi(S/k,\xi)\to 1?$$
 
\end{questions}

\subsection{The case of a complete discrete valuation ring} \label{complet.dvr}
Let $A=k[[t]]$ where $k$ is a perfect field, with quotient field $K=k((t))$.
Let $\mathfrak X$ be a smooth connected formal affine scheme over $\Spf A$. Let $X_0$ be the special fiber of $\mathfrak X$ and let $X$ be the generic fiber. 
Assume that $\mathfrak X$ admits an $A$-rational point $\xi$.  
Our aim is to show that the category $\str(\mathfrak X/A)$ is Tannakian.

We will show that $\str^{\rm o}(\mathfrak X/A)$ is subcategory of definition in $\str(\mathfrak X/A)$, that is, any stratified sheaf in $\str(\mathfrak X/A)$ is a quotient of a stratified bundle. 
First we need the following.
\begin{proposition}\label{formal_lift}
The restriction functor from $\str(\mathfrak X/k)$ to $\str(X_0/k)$ is an equivalence. In particular, any exact sequence in $\str(X_0/k)$ can be lifted to an exact sequence in $\str(\mathfrak X/k)$.\end{proposition}
\begin{proof}
If $k$ is of positive characteristic, this is a result of Gieseker \cite[Lemma~1.5]{Gi}. He constructed an explicit lift of a stratified bundle on $X_0/k$ to a stratified bundle on $\mathfrak X/k$ and showed that this lift yields a functor which is quasi-inverse to the restriction functor, thus giving the equivalence.

The case of zero characteristic can be proved using the method of Katz in the proof of \cite[Prop.~8.8]{Katz}. Let $M$ be a stratified module over $\Diff(\mathfrak X/k)$. 
The action of this algebra on $M$ will be denoted as usual by $\nabla$. Since $\mathcal O_{\mathfrak X}$ contains the field $k$, $\Diff(\mathfrak X/k)$ is generated by the derivations, that is a stratification is nothing but a flat connection.  We first show that it is locally free. By means of Proposition \ref{pro_santos} it suffices to show that $M$ is $t$-torsion free. This is a local property on $\mathfrak X$.
 
Let $(x_1,\ldots,x_r,t)$ be local coordinates on $\mathfrak X$. Thus $\partial_{x_i}$'s commute each other and commute with $\partial_t$, and we have
$$\partial_{x_i}(x_j)=\delta_{ij};\quad \partial_{x_i}(t)=\partial_{t}(x_i)=0;\quad \partial_{t}(t)=1.$$
  One considers the $k$-linear operator 
$\displaystyle P:=\sum_{i=0}^\infty\frac{(-t)^i}{i!}\partial_t^i.$
It has the following properties, cf. \cite[Sect.~8]{Katz}: $$P^2=P;\quad P(m)=m \text{  (mod } tM);\quad P(fm)=f(0)P(m),\quad f\in A, m\in M.$$ 
Hence, setting $M^{\nabla_t}:=\ker\nabla(\partial_t)$, we have
\begin{equation}\label{eq.Mt} M^{\nabla_t}=\im P;\quad M=M^{\nabla_t}\oplus tM; \quad M^{\nabla_t}\cong M_0:=M/tM.\end{equation}
Further the map $A\otimes M^{\nabla_t}\to M$,
$f\otimes m\mapsto fm$ is injective, in particular, $M^{\nabla_t}$ is $t$-torsion free.

Assume that $tm=0$ for some $m\in M$. If $m\neq 0$ it has a unique presentation 
$m=t^{k-1}(tm_1+m_0)$, where $k>0$ maximal, $m_0\in M^{\nabla_t}$ (this is due to the completeness of the $t$-adic topology on $\mathcal O_{\mathfrak X}$).  Then we have
$$0=\nabla(\partial_t)^k(tm)=\nabla(\partial_t)^k (t^{k+1}m_1)+k!m_0,$$
(since $\nabla(\partial_t)(m)=0$).
Hence $m_0\in tM$, which implies $m_0=0$, contradiction. Hence $m=0$. Thus $M$ is $t$-torsion free, hence locally free over $\mathcal O_{\mathfrak X}$ and consequently the restriction functor $M\mapsto M_0=M/tM$ is exact. It also implies that $\str(\mathfrak X/k)$ is an abelian rigid tensor category over $k$.

A section of $M$ (as an object of $\str(\mathfrak X/k)$) is horizontal iff (locally) it lies in $M^{\nabla_t}\cong M_0$ and is annihilated by $\nabla(\partial_{x_i})$, and hence iff its image in $M_0$ is a horizontal section of $M_0$ as an object of $\str(X_0/k)$. We conclude that the restriction functor $\str(\mathfrak X/k)\to \str(X_0/k)$ is faithful. 

Conversely, the third isomorphism in  \eqref{eq.Mt} shows that on an open $U$ of $X_0$ (which topologically homeomorphic to $X$), small enough so that local coordinates on it exist, a horizontal section of $M_0|_U$ can be uniquely lifted to a horizontal section of $M|_U$.  Let now $s_0$ be a horizontal section of $M_0\in \str(X_0/k)$. Consider an open covering $(U_\alpha)$ of $X_0$ such that on each $U_\alpha$ there exist local coordinates $(x_i,t)$. Let $s_{0,\alpha}$ be the restriction of $s_0$ on $U_\alpha$. Lift $s_{0,\alpha}$ to a horizontal section $s_\alpha$ of $M$ on $U_\alpha$. The restrictions of $s_\alpha$ and $s_\beta$ on $U_{\alpha\beta}$ agree as they are liftings of the same section. Hence the $s_\alpha$'s glue together to give a horizontal section of $M$ on $X$. Thus the restriction functor is full.

It remains to show that this functor is essentially surjective, that is, each stratified bundle $M_0$ on $X_0/k$ can be lifted to a stratified bundle on $\mathfrak X/k$. We first assume that on $X$ there exist global coordinates $(x_1,\ldots,x_n,t)$ and that $M_0$ is free over $\mathcal O_{X_0}$ with basis $(e_i^0)$. Given this, we will  show that a flat connection on $M_0$ can be lifted to $M=\langle e_i\rangle_{\mathcal O_{\mathfrak X}}$.

Consider the action of the operator $P$ defined above on the algebra $\mathcal O_{\mathfrak X}$. We have
($D:=\nabla(\partial_t)$)
\begin{eqnarray*}P(ab)&=&\sum_i\frac{(-t)^i}{i!}D^i(ab)\\
&=&\sum_i\frac{(-t)^i}{i!}\sum_j{i\choose j}D^j(a)D^{i-j}(b)\\
&=&\sum_j\frac{(-t)^j}{j!}D^j(a)\sum_{i\geq j}\frac{(-t)^{i-j}}{i!}D^{i-j}(b)\\
&=&P(a)P(b).
\end{eqnarray*}
Hence the isomorphism $\varphi:\mathcal O_{\mathfrak X}^{\nabla_t}\to \mathcal O_{X_0}$, induced by $P$, is an isomorphism of algebras. Notice that, since $[\partial_{x_i},\partial_t]=0$ we have $[\partial_{x_i},D]=0$, for all $i$. Thus $\varphi$ commutes with the action of $\partial_{x_i}$.

Let $\psi$ be the inverse of $\varphi$. Assume that the actions of $\nabla(\partial_{x_i})$ on the basis $(e^0_j)$ is given by a set of  matrices $a_i=(a^k_{ij})$:
$$\nabla(\partial_{x_i})(e^0_j)=\sum_ka_{ij}^ke^0_k.$$
The flatness of $\nabla$ is expressed in terms of the Maurer-Cartan equation involving the matices $(a^k_{ij})$ and their partial derivatives in $x_i$'s. This equation is preserved by $\psi$, which means we can lift them to a set of matrices $A_j=(A^k_{ij})$ such that the equations:
$$\nabla(\partial_{x_i})(e_j)=\sum_kA_{ij}^ke_k$$
defines a flat connection on $\mathfrak X/A$.

Finally we simply set   $\nabla(\partial_t)(e_i)=0$. It is straightforward to check that $\nabla(\partial_t)$ commutes with $\nabla(\partial_{x_i})$ using the fact that $\nabla(\partial_t)(A^k_{ij})=0$ as these elements lie in $A^{\nabla_t}$. Thus, we have constructed a flat connection on $M$.

In the general case we consider an open covering of $X_0$, such that on each open, the connection $M_0$ is free and local coordinates exist. Then on each open we can lift $M_0$. As the lift on each open is unique, they glue together to give a lift of $M_0$ on the whole $\mathfrak X$. 
\end{proof}

Notice that if a stratified sheaf $E_0$ on $\mathfrak X/A$ is annihilated by $t$ then it can be considered as a sheaf on $X_0/k$ and hence can be lifted to a stratified bundle on $\mathfrak X/k$, say $E$ and we have an exact sequence
\begin{equation} 0\to  E\stackrel{[t]}\to E\to E_0\to 0\label{eq:lift1}\end{equation}
where $[t]$ denotes the map multiplying by $t$.
Thus in this case $E_0$ is a quotient of $E$ (considered as stratified sheaf on $\mathfrak X/A$).

\begin{proposition}
Each object of $\str(\mathfrak X/A)$ is a quotient of an object of $\str(\mathfrak X/A)^0$. Consequently, $\str(\mathfrak X/A)$ is a Tannakian category.\end{proposition}
\begin{proof}
Let $E$ be an object of $\str(\mathfrak X/A)$. Then the subsheaf $E_t$ consisting of sections annihilated by some power of $t$ is invariant under the stratification. We have an exact sequence
\begin{equation}\label{eq:ext1}0\to E_t\to E\to E_{\rm fr}\to 0,\end{equation}
with $E_{\rm fr}$ a $t$-torsion free stratified sheaf (hence is locally free by the lemma above). There exists a least integer $r$ such that $E_t$ is annihilated by $t^r$. We will use induction on $r$.

For $r=1$, the subsheaf $tE\subset E$ is $t$-torsion free. Indeed, if a section $ts $ in $tE$ is torsion then $s$ is itself torsion, hence is annihilated by $t$. Consider the exact sequence
$$0\to tE\to E\to E/tE\to 0.$$
The sheaf $E/tE$ is in $\str(X_0/k)$, hence can be lifted to a stratified bundle $F$ on $\mathfrak X/k$: $F\to E/tE$. Pull back the above sequence along this map (considered as morphism in $\str(\mathfrak X/A)$), we get the following commutative diagram
$$\xymatrix{
0\ar[r]& tE\ar[r]& E\ar[r]&E/tE\ar[r]& 0\\
0\ar[r]& tE\ar[r]\ar[u]_=& \widehat E\ar[r]\ar[u]& F\ar[r]\ar@{->>}[u]& 0}$$
In particular, the map $\widehat E\to E$ is surjective. But $\widehat E$ is torsion free as the two sheaves
$tE$ and $F$ are locally free. Thus $\widehat E$ is the needed stratified bundle on $\mathfrak X/A$.
 
Let now $E$ be such that $E_t$, the subsheaf of torsion sections, is annihilated by $t^n$, $n>1$. Let $E_0$ be the subsheaf of $E$ of section annihilated by $t$. Then we have exact sequence
$$ 0\to E_0\to E\to E'\to 0,$$
where for $E'$ its torsion part $E'_t$ is annihilated by $t^{n-1}$. By induction we can lift $E'$ to $F'$ and hence, by pulling-back, we can lift $E$:
$$\xymatrix{
0\ar[r]& E_0\ar[r]& E\ar[r]&E'\ar[r]& 0\\
0\ar[r]& E_0\ar[r]\ar[u]_=& F\ar[r]\ar[u]& F'\ar[r]\ar[u]& 0}$$
Since $F'$ is locally free, $E_0$ is the torsion subsheaf of $F$ and we can lift $F$.\end{proof}

\begin{remarks}
Y. Andre has given in \cite[3.2.1.5]{Andre} an example showing that the differential Galois  group of a connection in $\str^{\rm o}(\mathfrak X/A)$ may be of infinite type over $A$. 
\end{remarks}

\begin{appendix}
\section{Tannakian duality for flat coalgebras over Dedekind rings}\label{App2}
\setcounter{subsubsection}{0}
\setcounter{subsection}{1}
In this appendix we give a quick, complete and self-contained proof of Theorem \ref{thSa}.  
First we will recall the notion ind-category of an abelian category. The two equivalent descriptions of the ind-category will play a crucial role in Saavedra's proof. A category $\mathcal{I}$ is called a filtered category if to every pair $i, j$ of objects in $\mathcal{I}$ there exists an object $k$
 such that ${\sf Hom}(i, k)$ and ${\sf Hom}(j, k)$ are both not empty, and for every pair $u, v : i \longrightarrow j,$ there exists a morphism $w: j \longrightarrow k $ such that $wu = wv.$

 \begin{definition}
{\it \textit{Ind-categories.}} Let $\mathcal{C}$ be an abelian category. The category ${\sf Ind}(\mathcal{C})$ consists of  functors $X: \mathcal{I} \longrightarrow \mathcal{C},$ where $\mathcal{I}$ is a filtering category. We usually denote $X_i$ for $X(i), i \in \mathcal{I},$ an write
$$X=\varinjlim_{i\in\mathcal I} X_i.$$
 For two objects $\displaystyle X= \varinjlim_{i\in \mathcal I}X_i$ and $\displaystyle Y = \varinjlim_{j\in\mathcal J}{Y_j } $   their hom-set is defined to be  
 $${\sf  Hom}(X,Y):= \displaystyle \varprojlim_{i\in \mathcal I} \displaystyle \varinjlim_{j\in\mathcal J} {\sf Hom}(X_i, Y_j).\eqno\Box$$
\end{definition}
 Let $\omega: \mathcal{C} \longrightarrow \mathcal{D}$ be a functor. 
The extension of $\omega, {\sf Ind}(\omega): {\sf Ind(\mathcal{C})} \longrightarrow \sf{Ind}(\mathcal{D})$ is defined by
$${\sf Ind}(\omega)(\displaystyle \varinjlim_i {X_i}) := \displaystyle \varinjlim_i {\omega(X_i)}.$$

There is an alternative description of ${\sf Ind}(\mathcal C)$. Denote ${\sf Lex}(\mathcal{C}^{op}, {\mathcal Sets})$ the category of left exact  functors from $\mathcal{C}^{op}$ to the category of sets. For $X= \varinjlim_i X_i$ we define functor 
$$\displaystyle \varinjlim_i h_{X_i}(-):= \displaystyle \varinjlim_i {\sf Hom}(-, X_i) \in {\sf Lex}(\mathcal{C}^{op}, {\mathcal Sets}).$$ 
This yields a functor
$ \sf{Ind(\mathcal{C})} \longrightarrow {\sf Lex}(\mathcal{C}^{op}, {\mathcal Sets})$
which is an equivalence (cf.\cite{SGA4}, I.8.3.3). Recall that the Hom-sets for objects of ${\sf Lex}(\mathcal{C}^{op}, {\mathcal Sets})$ are by definition the sets of natural transformations. For simplicity, we shall use the notation ${\sf Hom}(F,G)$ instead of ${\sf Nat}(F,G)$ for objects of this category.
 
\subsubsection{}  Suppose that $\mathcal{C}$ is an $R$-linear {\em Noetherian abelian} category. Let ${\sf Lex}_R(\mathcal{C}^{op}, {\sf Mod}(R))$ be category of  $R$-linear left exact  functors from $\mathcal{C}^{op}$ to the category of  modules ${\sf Mod}(R)$. Then the natural functor
$$ {\sf Lex}_R(\mathcal{C}^{op}, {\sf Mod}(R)) \xrightarrow{\ \simeq\ }{\sf Lex}(\mathcal{C}^{op}, {\mathcal Sets}) $$ 
is an equivalence  (cf. Gabriel  \cite[II]{Ga}). 
Thus, for an $R$-linear Noetherian abelian category  we have an equivalence
$${\sf Ind}(\mathcal{C})\simeq {\sf Lex}_R(\mathcal{C}^{op}, {\sf Mod}(R)),\quad X=\varinjlim_iX_i\longmapsto \displaystyle \varinjlim_i h_{X_i}(-).$$
 Further the category ${\sf Ind}(\mathcal{C})$ is locally Noetherian and the inclusion $\mathcal C\to {\sf Ind}(\mathcal{C})$ identifies $\mathcal C$ with the full subcategory of Noetherian objects in ${\sf Ind}(\mathcal{C})$, \cite[II,4, Thm.1]{Ga}.
 
 The following are our main examples.
 \begin{example} The category ${\sf Mod}_{\rm f}(R)$ of finitely generated $R$-modules, where $R$ is a Noetherian ring, is a Noetherian category. Its Ind category is precisely the category ${\sf Mod}(R)$ of all $R$-modules. This is obvious.\end{example}
 
\begin{example}\label{ex_comod} Let $L$ be a coalgebra over a commutative ring $R$. Denote by ${\sf Comod}(L)$  the category of right $L$-comodules and by ${\sf Comod}_{\rm f}(L)$ the subcategory of comodules which are finitely generated as $R$-module.  Then:
\begin{itemize}\item[(i)]
 If $L$ is flat over $R$ then ${\sf Comod}(L)$ is an abelian category. In fact, the flatness of $L$ implies that the kernel of a homomorphism of $L$-comodules is equipped with a natural coaction of $L$. In particular, the forgetful functor from ${\sf Comod}(L)$ to ${\sf Mod}(R)$ is exact.  The converse is also true: if the forgetful functor preserves kernels then $L$ is flat over $R$. 
 \item[(ii)] Assume that $L$ is flat over $R$ and $R$ is Noetherian. 
According to Serre \cite[Cor.~2]{Se}  each $L$-comodule is the union of its $R$-finite subcomodules. Consequently, ${\sf Comod}(L)$ is locally Noetherian and $\comodf(L)$ is the full subcategory of Noetherian objects.
 \end{itemize}
\end{example}

Let $\mathcal{C}$ be an $R$-linear abelian category, and $\omega : \mathcal{C} \longrightarrow \modf(R)$  be an $R$-linear exact faithful functor. Suppose that there exists a full subcategory of definition $\mathcal{C}^{\rm o}$ in $\mathcal{C}.$ 
Our aim is to show that there exists a flat $R$-coalgebra $L$ such that  $\omega$ induces an equivalence between  $\comodf(L)$ and $\mathcal{C}$, and between ${\sf Comod}(L)$ and ${\sf Ind}(\mathcal{C})$.

The functor  $\omega$ induces a functor 
$  {\sf Ind}( \mathcal C)\longrightarrow {\sf Mod}(R),$ 
which we, by abuse of language, will denote simply by $\omega$.
Recall that we identify ${\sf Ind}(\mathcal{C})$ with ${\sf Lex}(\mathcal{C}^{op}, {\sf Mod}(R))$, the category of left exact functors on $\mathcal{C}^{op}$ with values in ${\sf Mod}(R)$. The key technique is to use alternatively these two equivalent descriptions of one category.

\subsubsection{}
For  any $R$-algebra $A$, we define functor 
$$ F^A: \mathcal{C}^{op} \longrightarrow {\sf Mod}(A),\quad X \longmapsto {\sf Hom}( \omega(X), A).$$
 Then $F^A$ is an object of $ {\sf Lex}(\mathcal{C}^{op}, {\sf Mod}(R))$.
Set $F:=F^R$. There is a natural $A$-linear transformation $A\otimes F\to F^A$:
$$\theta_X:A\otimes {\sf Hom}(\omega(X),R)\to {\sf Hom}(\omega(X),A),\quad a\otimes f\longmapsto af.$$
   
\begin{lemma} \label{lem_restriction}The $A$-linear transformation $\theta:A\otimes F\to F^A$ given above is an isomorphism. 
\end{lemma}
\begin{proof}   For any $K, G \in {\sf Lex}(\mathcal{C}^{op}, {\sf Mod}(R))$ we denote $K^{\rm o}, G^{\rm o}$ their restrictions to $(\mathcal{C}^{\rm o})^{op}$, respectively. 
We claim that
\begin{eqnarray}\label{eq1}
{\sf Hom}(K,G) \simeq {\sf Hom}(K^{\rm o}, G^{\rm o}).
\end{eqnarray}
Indeed, let $\theta\in {\sf  Hom}(K^{\rm o}, G^{\rm o})$, that is we have a family $\theta_X:K^{\rm o}(X)\to G^{\rm o}$ for $X\in\mathcal C^{\rm o}$ commuting with morphism in $\mathcal C^{\rm o}$. Since each object of $\mathcal C$ can be represented as a cokernel of a morphism $X_1\to X_2$ in $\mathcal C^{\rm o}$, we see that $\theta$ extends uniquely to a natural transformation $K\to G$ (as these functors are left exact on $\mathcal C^{\rm op}$).

For $X\in \mathcal  C^{\rm o}$, $\omega(X)$ is finite projective over $R$, hence  
$$ F^A(X) ={\sf Hom}(\omega(X), A) \simeq {\sf Hom}(\omega(X), R) \otimes A = A \otimes F(X).$$
Therefore, for any $G\in {\sf Lex}(\mathcal{C}^{op}, {\sf Mod}(R))$, we have
\begin{eqnarray}
{\sf  Hom}((F^A)^{\rm o}, G^{\rm o}) = {\sf  Hom}((A\otimes F)^{\rm o}, G^{\rm o})
\end{eqnarray}
 and \eqref{eq1}   yields
\begin{eqnarray}
 {\sf  Hom}(F^A, G) = {\sf  Hom}(A\otimes F, G).
\end{eqnarray}
 So we have $F^A\simeq A \otimes F.$
\end{proof}

We will show that $L:=\omega(F)$ is the coalgebra to be found. To show this, first we will need
\begin{lemma}
 For any  $X\in{\sf Lex}(\mathcal C^{op},{\sf Mod}(R))$ and $R$-algebra $A$ we have the following $A$-linear isomorphism:
\begin{eqnarray}\label{eq2}
{\sf Hom}(X, F^A) \simeq 
 {\sf Hom}_A(A\otimes \omega(X), A)={\sf Hom}_R( \omega(X), A).
\end{eqnarray}
\end{lemma}
 \begin{proof}  Every $X \in {\sf Lex}(\mathcal C^{op},{\sf Mod}(R))$ can be represented as 
 $X= \displaystyle \varinjlim_i h_{X_i}  ( X_i \in \mathcal{C}),$ where $h_{X_i}$ is a functor  over $\mathcal{C}$,  defined by $h_{X_i}(-) := {\sf Hom}_{\mathcal{C}}( -, X_i).$ Hence we have
  \begin{eqnarray*}
  {\sf Hom}(X, F^A) &=& {\sf Hom}( \varinjlim h_{X_i},F^A)\\
                                            & = &\varprojlim {\sf Hom}(h_{X_i}, F^A)\\
                                            & = &\varprojlim F^A (X_i) \\
                                            & = &{\sf Hom}_R( \varinjlim\omega (X_i), A)\\
                                             &= &{\sf Hom}_R(\omega(\varinjlim h_{X_i}), A)\\
                                              &=& {\sf Hom}_R( \omega(X), A).
\end{eqnarray*}
It is easy to see that all isomorphisms are $A$-linear.
\end{proof}
Isomorphism \eqref{eq2} for $A=R$ and $X=F$  reads ${\sf Hom}(F,F)\simeq {\sf Hom}_R(\omega(F),R)$.
We denote $L:=\omega(F)$ and let $\varepsilon:L\to R$ be the map on the right hand side that corresponds to the identity transformation on the left hand side of this isomorphism. The next lemma shows that one can replace the algebra $A$ in \eqref{eq2} by any $R$-module $M$ to get $R$-linear isomorphisms.
\begin{lemma}\label{lem_XM}
There exists a natural $R$-linear isomorphism extending \eqref{eq2}
\begin{equation}\label{eq3}
\Phi_{X,M}:  {\sf Hom}(X,M\otimes F) \simeq   {\sf Hom}_R(\omega(X), M),
\end{equation}
which is given explicitly by 
$$\Phi_{X,M}(f)=  ({\sf id}_M\otimes\varepsilon)\circ \omega(f).$$
\end{lemma}
\begin{proof}
 For any $R$-module $M$, we can make $R\oplus M$ into an $R$-algebra by letting $M$ be an ideal with square null. Hence the isomorphism \eqref{eq3} is a direct consequence of \eqref{eq2}. 
 By definition $\Phi_{F,R}$ is given by 
 $$\Phi_{F,R}(f)=\varepsilon\circ \omega(f).$$ 
Each $R$-linear map $\iota:R\to M$ induces by functoriality the commutative diagram 
$$\xymatrix{
{\sf Hom}(F,F)\ar[d]_{(\iota\otimes {\sf id}_F)\circ -}\ar[rr]^{\varepsilon\circ \omega(-)}&& {\sf Hom}_R(\omega (F),R)\ar[d]^{\iota\circ -}\\
{\sf Hom}(F,M\otimes F)\ar[rr]_{\Phi_{F,M}}&&
{\sf Hom}_A(\omega (F),M)}$$
Now, the identity on $F$ yields the equality: $$\iota\circ\varepsilon=\Phi_{F,M}(\iota\otimes{\sf id}_{\omega(F)}):\omega(F)\to M.$$
Hence, for $m=\iota(1)$, we have  $\Phi_{F,M}(l)=\varepsilon(l)m$,  $l\in\omega(F)$. 
Thus the claim holds for $X=F$. Since the $\omega$ and Hom-functor in the first variant commute with direct limits we conclude that the claim hold of $X=N\otimes F$ for any $R$-module $N$. Now the general case follows from the following diagram
$$\xymatrix{ {\sf Hom}(M\otimes F,M\otimes F)\ar[r]^{\Phi_{F,M\otimes F}}\ar[d]_{(-)\circ f}& {\sf Hom}(M\otimes\omega(F),M)\ar[d]^{(-)\circ\omega(f)}\\
{\sf Hom}(X,M\otimes F)\ar[r]_{\Phi_{X,M}}&{\sf Hom}(\omega(X), M)}$$
applied for the identity of $M\otimes F$:
$$\Phi_{X,M}(f)=\Phi_{F,M}(\id)\circ\omega(f)=(\id_M\otimes\varepsilon)\circ\omega(f).$$
\end{proof}

\begin{proposition}\label{pro_reconstruction}
Let $L:=\omega(F)$. Then it is a coalgebra with $\varepsilon$ being the counit and $\omega$ factors though a functor
$${\sf Ind}(\mathcal C)\to {\sf Comod}(L).$$
\end{proposition}
\begin{proof}
Choose $M= \omega(X)$ in \eqref{eq3} we have a morphism $ {\sigma}_X :X \longrightarrow \omega (X) \otimes F$ which
corresponds to the identity element $\id_{\omega(X)}$ under the isomorphism $\Phi_{X,\omega(X)}$ of  Lemma \ref{lem_XM}, thus we have
\begin{equation}\label{eq_counit}
	({\sf id}_{\omega(X)} \otimes \varepsilon) \circ \omega(\sigma_{X}) = {\sf id}_{\omega(X)}.
\end{equation}

 For any morphism $\lambda: X \longrightarrow Y$ in ${\sf Ind}( \mathcal C)$, according to \ref{lem_XM} we have
 the following equalities:
 \begin{eqnarray*}&&
\Phi_{X,\omega(Y)}\left( (\omega(\lambda) \otimes {\sf id}_ {F}) \circ \sigma_X\right)=\omega(\lambda),\\
&&\Phi_{X,\omega(Y)}\left(\sigma_Y \circ \lambda\right)= \omega(\lambda).\end{eqnarray*}
Thus $(\omega(\lambda) \otimes {\sf id}_{F}) \circ \sigma_X= \sigma_Y \circ \lambda,$ i.e, the following diagram commutes:
\begin{equation}\label{eq5}
\xymatrix{
       X  \ar[rr]^{\lambda} \ar[d]_{\sigma_X}  && Y\ar[d]^{\sigma_Y} \\
\omega(X) \otimes F \ar[rr]^{\omega(\lambda) \otimes {\sf id}_{F}} & &\omega(Y) \otimes F.  }\end{equation}
For $Y=\omega(X)\otimes F$ and $\lambda=\sigma_X$, we get
\begin{equation}\label{eq5b}
\xymatrix{
X  \ar[d]_{\sigma_{X}}  \ar[rr]^{\sigma_{X}}&& \omega (X) \otimes F\ar[d]^{{\sf id}_{\omega(X)}
\otimes\sigma_{ F}} \\
 \omega (X) \otimes F \ar[rr]_{\omega(\sigma_X)\otimes \sf id_F} && \omega (X)  \otimes L \otimes F}
\end{equation}
Applying $\omega$ on this diagram we obtain a commutative diagram in ${\sf Mod}(R)$:
 \begin{equation}\label{eq5c}
\xymatrix{
 \omega (X) \ar[d]_{\omega(\sigma_X)}  \ar[rr]^{\omega(\sigma_X)}&& \omega (X) \otimes L\ar[d]^{\sf id\otimes \Delta} \\
 \omega (X) \otimes L \ar[rr]_{\omega(\sigma_X)\otimes{\sf id
_L}}& & \omega (X)  \otimes L \otimes L,}
\end{equation} 
where $\Delta:=\omega(\sigma_F)$. Together with \eqref{eq_counit}, this diagram for $X=F$ gives a coalgebra structure on $L$ with $\Delta$ being the coproduct and hence, for any $X$, it gives a comodule structure of $L$ on $\omega(X)$.
\end{proof}

\begin{proof} (of Theorem \ref{thSa}) Let $L$ be defined as in Proposition \ref{pro_reconstruction}.
We consider $\omega$ as a functor $\mathcal C\to \comodf(L)$. It is to show that $\omega$ is an equivalence of category. By definition it is faithful.   To see the fullness, suppose $X, Y \in \mathcal C$ and
 $\alpha : \omega(X) \longrightarrow \omega(Y)$  is a homomorphism of  $L $-comodules, i.e., we have
 $$(\alpha\otimes{\sf id})\circ\omega(\sigma_X)=\omega(\sigma_Y)\circ \alpha:\omega(X)\to \omega(Y)\otimes L.$$ 
Then
 $
\xymatrix{\omega(X)  \ar[r]^{\alpha}  &  \omega(Y)    \ar[r]^{\omega(\sigma_Y)\quad} & \omega(Y) \otimes L }
$
is  the image under $\omega$ of the morphism
$$\xymatrix{ X\ar[r]^{\sigma_X\quad}& \omega(X) \otimes {F} \ar[r]^{\alpha \otimes {\sf id}_{F}}&\omega(Y) \otimes F.}$$

\let\to\longrightarrow

Notice that \eqref{eq5c} (for $X$ replaced by $Y$) yields a split exact sequence 
\begin{equation}\label{eq_split}\xymatrix{0\ar[r]&\omega(Y)\ar[r]^{\omega(\sigma_Y)} &\omega (Y) \otimes L\ar[r]^{\delta\quad}&\omega (Y)  \otimes L \otimes L,}\end{equation}
where the second homomorphism is 
$\delta={\sf id}\otimes\Delta-\omega(\sigma_X)\otimes{\sf id},$
and the splitting is given by ${\sf id}\otimes\varepsilon:\omega(Y)\otimes L\to \omega(Y).$
 This sequence is the similar image under $\omega$ of the sequence coming from \eqref{eq5b}:
$$0\to Y\to  \omega(Y)  \otimes F\to \omega(Y) \otimes L \otimes F.$$
Hence the latter sequence is also exact. On the other hand, it follows from the faithfulness of $\omega$ that the composed map
$$\xymatrix{X\ar[r]^{\sigma_X\quad}& \omega(X)\otimes F\ar[r]^{\alpha\otimes \sf id}& \omega(Y)\otimes F\ar[r]&\omega(Y)\otimes L\otimes F}$$
is the zero morphism (since its image under $\omega$ is zero by means of \eqref{eq5c} and the fact that $\alpha$ is a homomorphism of $L$-comodules). Consequently, the morphism
$\xymatrix{X\ar[r]^{\sigma_X\quad}& \omega(X)\otimes F\ar[r]^{\alpha\otimes \sf id}& \omega(Y)\otimes F}$ factor through a morphism $f:X\to Y$ and the morphism $\sigma_Y$. Applying $\omega$ on the composition of these maps we conclude $\omega(f)=\alpha$, as $\omega(\sigma_Y)$ is injective. Thus $\omega$ is full.

It remains to show that $\varphi$ is essentially surjective. For any $L$-comodule $(E,\rho_E)$ let $E^{\rm o} \in  C$ be such that the sequence
$$0\to E^{\rm o}\to  E  \otimes F\xrightarrow{\quad\delta} E \otimes L \otimes F.$$
is exact, where $\delta=\rho_E\otimes{\sf id}-{\sf id}\otimes\sigma_F$.
Applying $\omega$ to this sequence and comparing with \eqref{eq_split} we conclude that $\omega(E^{\rm o})=E$.
 
Thus $\omega:\mathcal C\to \comodf(L)$ is an equivalence of categories. Thus the forgetful functor $\comodf(L)\to {\sf Mod}(R)$ is exact, hence $L$ is flat over $R$.
\end{proof}

\begin{remarks} \label{coend} (i)
Under  the equivalence of Theorem \ref{thSa}, $L$, with the right coaction of itself given by the coproduct, corresponds to $F$. Indeed, this follows from the natural isomorphism
$${\sf Hom}^L(E,L)\simeq {\sf Hom}_R(E,R),\quad f\mapsto \varepsilon\circ f.$$

(ii) There is another way to determine $L$ from the category of its comodules as follows.  We claim that there is a natural isomorphism
\begin{equation}\label{eq_coend}{\sf Nat}(\omega,\omega\otimes M)\simeq {\sf Hom}_R(L,M),\end{equation}
for any $R$-module $M$. Indeed, we have 
$${\sf Hom}_R(L,M)\simeq {\sf Hom}(F,F\otimes M)\simeq {\sf Hom}({\sf Hom}(\omega,R),{\sf Hom}(\omega,R)\otimes M).$$
By means of \eqref{eq1}, it suffices to show the isomorphism
$${\sf Nat}(\omega(X),\omega(X)\otimes M)\simeq {\sf Hom}({\sf Hom}(\omega(X),R),{\sf Hom}(\omega(X),R)\otimes M)$$
for any $X\in \mathcal C^{\rm o}$. Since for such $X$, $\omega(X)$ is finitely generated projective over $R$, the last isomorphism is obvious.
$L$ is usually referred to as the Coend of $\omega$, denoted $\coend (\omega)$.

(iii) If $\mathcal C=\comodf(L)$ and $\omega$ is the forgetful functor from $\mathcal C$ to ${\sf Mod}(R)$, then the isomorphism \eqref{eq_coend} implies that $\coend (\omega)\simeq L$. Thus a flat coalgebra over $R$ can be {\em reconstructed} from the category of its comodules.\hfill $\Box$\end{remarks}
\begin{remarks}\label{rmk2.11}
Let $(\mathcal C,\omega)$ and $(\mathcal C',\omega')$ be two categories satisfying the condition of Theorem \ref{thSa} and let $\eta:\mathcal C\to \mathcal C'$ be an $R$-linear functor such that $\omega'\eta=\omega$. Then $\eta$ induces a coalgebra homomorphism $f:L\to L'$. This can be seen from \eqref{eq_coend} as follows. The coaction of $L'$ on $\omega'(X')$ defines a natural transformation $\delta':\omega'\to \omega'\otimes L'$. Combine this with $\eta$ we obtain a natural transformation $\delta:\omega\to\omega\otimes L'$. Thus \eqref{eq_coend} yields a linear map $L\to L'$, which satisfies the following commutative diagram:
$$\xymatrix{
\omega(X)\ar[r]^{\delta}\ar[rd]_{\delta'}&\omega(X)\otimes L\ar[d]^{{\sf id}\otimes f}\\
&\omega(X)\otimes L'} $$ $$ \eqno\Box$$ 
\end{remarks}

\end{appendix} 
\section*{Acknowledgment}
The second named author would like to thank H. Esnault and J.P. dos Santos for their interests in the work and very helpful discussions. He would also like to express his gratitude to J.-P. Serre for explaining him about flat coalgebras.

\end{document}